\documentclass[sts,a4paper,preprint,reqno,11pt]{my_imsart}
\RequirePackage[OT1]{fontenc}
\usepackage{amsthm,amsmath}
\RequirePackage[colorlinks,citecolor=blue,urlcolor=blue,breaklinks=true]{hyperref}

\usepackage{amssymb,a4wide}
\usepackage{amsfonts,mathrsfs}
\usepackage{url}
\usepackage{pifont,dsfont}
\usepackage{breakcites,microtype}
\usepackage[round,sort]{natbib}

\usepackage{cleveref}
\usepackage{autonum}

\usepackage{nicefrac}
\usepackage{graphicx}

\usepackage{color}

\startlocaldefs

\newtheorem*{theorem*}{Theorem}

\newtheorem{rem}{Remark}[section]

\newtheorem{lemma}{Lemma}
\newtheorem{cor}{Corollary}
\newtheorem{prop}{Proposition}

\newtheorem{theorem}{Theorem}

\newcommand{\sgn}{\;\!{\rm sign}}

\newcommand{\R}{\mathbb R}
\newcommand{\RR}{\mathbb R}
\newcommand{\prob}{\mathbf{P}}
\newcommand{\esp}{\mathbb{E}}

\newcommand{\e}{\varepsilon}

\newcommand{\Cov}{\mathop{\bf Cov}\nolimits}

\newcommand{\bbetaz}{\boldsymbol{\beta^\star}}
\newcommand{\betaz}{\boldsymbol{\beta^\star}}

\newcommand{\barbbeta}{\;\bar{\!\!\boldsymbol{\beta}}}
\newcommand{\hatbbeta}{\,\hat{\!\boldsymbol{\beta}}}

\newcommand{\calP}{\mathcal{P}}
\newcommand{\calX}{\mathcal{X}}
\newcommand{\calC}{\mathcal{C}}

\newcommand{\bfA}{\mathbf{A}}
\newcommand{\bfB}{\mathbf{B}}
\newcommand{\bfC}{\mathbf{C}}
\newcommand{\bfI}{\mathbf{I}}

\newcommand{\bfZ}{\mathbf{Z}}
\newcommand{\bfX}{\mathbf{X}}
\newcommand{\bfU}{\mathbf{U}}
\newcommand{\bfM}{\mathbf{M}}
\newcommand{\bfV}{\mathbf{V}}
\newcommand{\bfD}{\mathbf D}
\newcommand{\bfW}{\mathbf{W}}

\newcommand{\bbfX}{\mathcal X}

\newcommand{\mca}{\mathcal{A}}
\newcommand{\mcM}{\mathcal{M}}
\newcommand{\mcb}{\mathcal{B}}
\newcommand{\mce}{\mathcal{E}}

\newcommand{\bs}{\boldsymbol}

\newcommand{\bd}{\bs d}
\newcommand{\bu}{\bs u}

\newcommand{\bx}{\bs x}

\newcommand{\by}{\bs y}

\newcommand{\bbeta}{\bs\beta}

\newcommand{\bxi}{\bs\xi}

\newcommand{\btheta}{\bs\theta}

\newcommand{\bEWA}{\hat\bbeta{}^{\text{\rm EWA}}}
\newcommand{\bfBEWA}{\hat\bfB{}^{\text{\rm EWA}}}

\newcommand{\EWA}{\hat\beta^{\text{\rm EWA}}}
\newcommand{\bbetaLS}{\hat\bbeta{}^{\text{\rm LS}}}
\newcommand{\betaLS}{\hat\beta{}^{\text{\rm LS}}}

\newcommand{\tr}{{\rm Tr}}
\newcommand{\rk}{{\rm rank}}

\newcommand{\bfSigma}{\mathbf \Sigma}

\newcommand{\bfPi}{\mathbf \Pi}

\newcommand{\bfBz}{\mathbf{B^\star}}
\newcommand{\hatbfB}{\hat{\mathbf{B}}}
\newcommand{\barbfB}{\bar{\mathbf{B}}}

\newcommand{\bbfM}{\mathcal{M}_{m_1, m_2}}

\definecolor{newgreen}{rgb}{0.0, 0.5, 0.3}

\def\hat{\widehat}

\makeatletter
\let\orgdescriptionlabel\descriptionlabel
\renewcommand*{\descriptionlabel}[1]{%
    \let\orglabel\label
    \let\label\@gobble
    \phantomsection
    \protected@edef\@currentlabel{#1}
    \let\label\orglabel
    \orgdescriptionlabel{#1}%
}
\makeatother

\endlocaldefs

\begin{document}

\begin{frontmatter}

\title{On the Exponentially Weighted Aggregate with the Laplace Prior}
\runtitle{EWA with Laplace Prior}

\begin{aug}
\author{\fnms{Arnak S.} \snm{Dalalyan,}\ead[label=e2]{arnak.dalalyan@ensae.fr}}
\author{\fnms{Edwin} \snm{Grappin}\ead[label=e3]{edwin.grappin@ensae.fr}}
\and
\author{\fnms{Quentin} \snm{Paris}\ead[label=e4]{quentin.paris@gmail.com}}

\runauthor{A.\ Dalalyan et al.}

\affiliation{CREST, ENSAE, Universit\'e Paris-Saclay, National Research University - Higher School of Economics
and Laboratory of Stochastic Analysis and its Applications}

\address{3 avenue Pierre Larousse,
92245 Malakoff, France.\hfill\break
26 Shabolovka street, 119049 Moscow, Russia.}
\end{aug}

\begin{abstract}
In this paper, we study the statistical behaviour of the Exponentially Weighted Aggregate (EWA) in the problem of
high-dimensional regression with fixed design. Under the assumption that the underlying regression vector is sparse,
it is reasonable to use the Laplace distribution as a prior. The resulting estimator and, specifically, a particular
instance of it referred to as the Bayesian lasso, was already used in the statistical literature because of its
computational convenience, even though no thorough mathematical analysis of its statistical properties was carried out.
The present work fills this gap by establishing sharp oracle inequalities for the EWA with the Laplace prior. These
inequalities show that if the temperature parameter is small, the EWA with the Laplace prior satisfies the same type
of oracle inequality as the lasso estimator does, as long as the quality of estimation is measured by the prediction
loss. Extensions of the proposed methodology to the problem of prediction with low-rank matrices are considered.
\end{abstract}

\begin{keyword}[class=MSC]
\kwd[Primary ]{62J05}
\kwd[; secondary ]{62H12}
\end{keyword}

\begin{keyword}
\kwd{sparsity}
\kwd{Bayesian lasso}
\kwd{oracle inequality}
\kwd{exponential weights}
\kwd{high-dimensional regression}
\kwd{trace regression}
\kwd{low-rank matrices}
\end{keyword}

\end{frontmatter}

\section{Introduction}

We investigate statistical properties of the Exponentially Weighted Aggregate (EWA) in the context of
high-dimensional linear regression with fixed design and under the sparsity scenario.
This corresponds to considering data that consist of $n$ random observations $y_{1},\ldots,y_{n}\in\R$ and $p$
fixed covariates $\bx^{1},\ldots,\bx^{p}\in \R^{n}$. We further assume that there is a vector $\bbeta^{\star}\in\R^p$
such that the residuals $\xi_i=y_{i}-\beta^{\star}_{1}\bx^{1}_{i}-\ldots-\beta^{\star}_{p}\bx^{p}_{i}$ are independent,
zero mean random variables. In vector notation, this reads as
\begin{equation}
\label{model}
\by=\bfX\bbeta^{\star}+\bxi,
\end{equation}
where $\by=(y_{1},\ldots, y_{n})^{\top}$ is the response vector, $\bfX=(\bx^{1},\ldots, \bx^{p})\in\R^{n\times p}$
is the design matrix and $\bxi$ is the noise vector. For simplicity, in all mathematical results, the noise vector is assumed to
be distributed according to the Gaussian distribution ${\mathcal N}(0,\sigma^2\bfI_{n})$. We are mainly interested in obtaining
mathematical results that cover the high-dimensional setting. This means that our goal is to establish risk bounds that
can be small even if the ambient dimension $p$ is large compared to the sample size. In order to attain this goal, we will
consider the, by now, usual sparsity scenario. In other words, the established risk bounds are small if the underlying
large vector $\bbeta^*$ is well approximated by a sparse vector. Note that this setting can be extended to the matrix case,
sometimes termed trace-regression \citep{Rohde11,KLT11}. Indeed, if the rows $\bx_1,\ldots,\bx_n$ of the design matrix $\bfX$
are replaced by $m_1\times m_2$ matrices $\bfX_1,\ldots,\bfX_n$, then the regression vector $\bbeta^\star$ is replaced by a
$m_1\times m_2$ matrix $\bfB^\star$ and the model of trace regression is
\begin{equation}
\label{model-tr}
y_i=\tr(\bfX_i^\top\bfB^\star)+\xi_i,\qquad i=1,\ldots,n.
\end{equation}
Our focus here is on the statistical properties related to the prediction risk. The important questions of variable selection
and estimation in various norms are beyond the scope of the present work.

In the aforementioned vector- and trace-regression models, the most thoroughly studied statistical procedures of estimation and
prediction rely on the principle of penalised least squares\footnote{Or, more generally, on the penalised empirical risk minimisation}.
In the vector-regression model, assuming that the
quadratic loss is used, this corresponds to analysing the properties of the estimator
\begin{equation}
\label{PLS}
\hat\bbeta^{\rm PLS} \in\text{arg}\min_{\bbeta\in \R^{p}} \bigg\{\frac1{2n}\sum_{i=1}^n (y_i-\bx_i^\top\bbeta)^2 + \lambda\, \text{Pen}(\bbeta)\bigg\},
\end{equation}
where $\lambda>0$ is a tuning parameter and $\text{Pen} : \R^p\to\R$ is a sparsity promoting penalty function. The literature
on this topic is so rich that it would be impossible to cite here all the relevant papers. We refer the interested reader to
the books \citep{bookBuhlman2011,koltchinskii2011,G15,vandeGeer2016} and the references therein. Among the sparsity promoting penalties,
one can mention the $\ell_0$ penalty (which for various choices of $\lambda$ leads to the BIC \citep{BIC}, the AIC \citep{AIC}
or to Mallows's Cp \citep{Mallows}), the $\ell_1$ penalty or the lasso \citep{Tib96}, the $\ell_q$ (with $0<q<1$) or
the bridge penalty \citep{Bridge,Bridge1}, the SCAD \citep{SCAD}, the minimax concave penalties \citep{MC+}, the entropy
\citep{Kolt09}, the SLOPE \citep{SLOPE1,SLOPE2}, etc.

The aggregation by exponential weights is an alternative approach to the problems of estimation and prediction that, roughly speaking, replaces
the minimisation by the averaging. Assuming that every vector $\bbeta\in\R^p$ is a candidate for estimating the true vector $\bbeta^\star$,
aggregation (cf., for instance, the survey \citep{TsybICM}) consists in computing a weighted average of the candidates. Naturally,
the weights are to be chosen in a data-driven way. In the case of the exponentially weighted aggregate (EWA), the weight
$\hat\pi_n(\bbeta)$ of each candidate vector $\bbeta$ has the exponential form
\begin{equation}
\label{potential}
\hat\pi_n(\bbeta)\propto \exp\big(-V_n(\bbeta)/\tau\big), \qquad
\text{where}\quad V_n(\bbeta)=\frac1{2n}\sum_{i=1}^n (y_i-\bx_i^\top\bbeta)^2 + \lambda\, \text{Pen}(\bbeta)
\end{equation}
is the potential used above for defining the penalised least squares estimator and $\tau>0$ is an additional tuning
parameter referred to as the temperature. Using this notation, the EWA is defined by
\begin{equation}
\label{EWALP}
\bEWA =\int_{\R^p}\bbeta\,\hat\pi_n(\bbeta)\,{\rm d}\bbeta.
\end{equation}
Exponential weights have been used for a long time  in statistical learning theory (cf., for instance, \cite{Vovk}).
Their use in statistics was initiated by Yuhong Yang in \citep{Yang6,Yang5,Yang4,Yang3} and by Olivier Catoni in a series
of preprints, later on included in \citep{Catoni2,Catoni1}.
Precise risk bounds for the EWA in the model of regression with fixed
design have been established in \citep{Leung2006, DT07, DT08, DT12b, DS12, DRXZ, Golubev1, Golubev2}. In the model of
regression with random design, the counterpart of the EWA, often referred to as mirror averaging, has been thoroughly
studied in \citep{YNTV,JRT, Audibert, Chesneau, Gaiffas, DT12a, Lecue}. Note that when the temperature $\tau$ equals
$\sigma^2/n$, the EWA coincides with the Bayesian posterior mean in the regression model with Gaussian noise provided
that the prior is defined by $\pi_0(\bbeta)\propto \exp(-\lambda\,\text{Pen}(\bbeta)/\tau)$. Thanks to this
analogy, we will call $\hat\pi_n$ pseudo-posterior density. Let us mention here that, considering
the path $\tau\mapsto \bEWA$ for $\tau\in(0,\sigma^2/n]$, we get a continuous interpolation between the penalised least
squares and the Bayesian posterior mean.

\begin{figure}
	\includegraphics[width = 150pt]{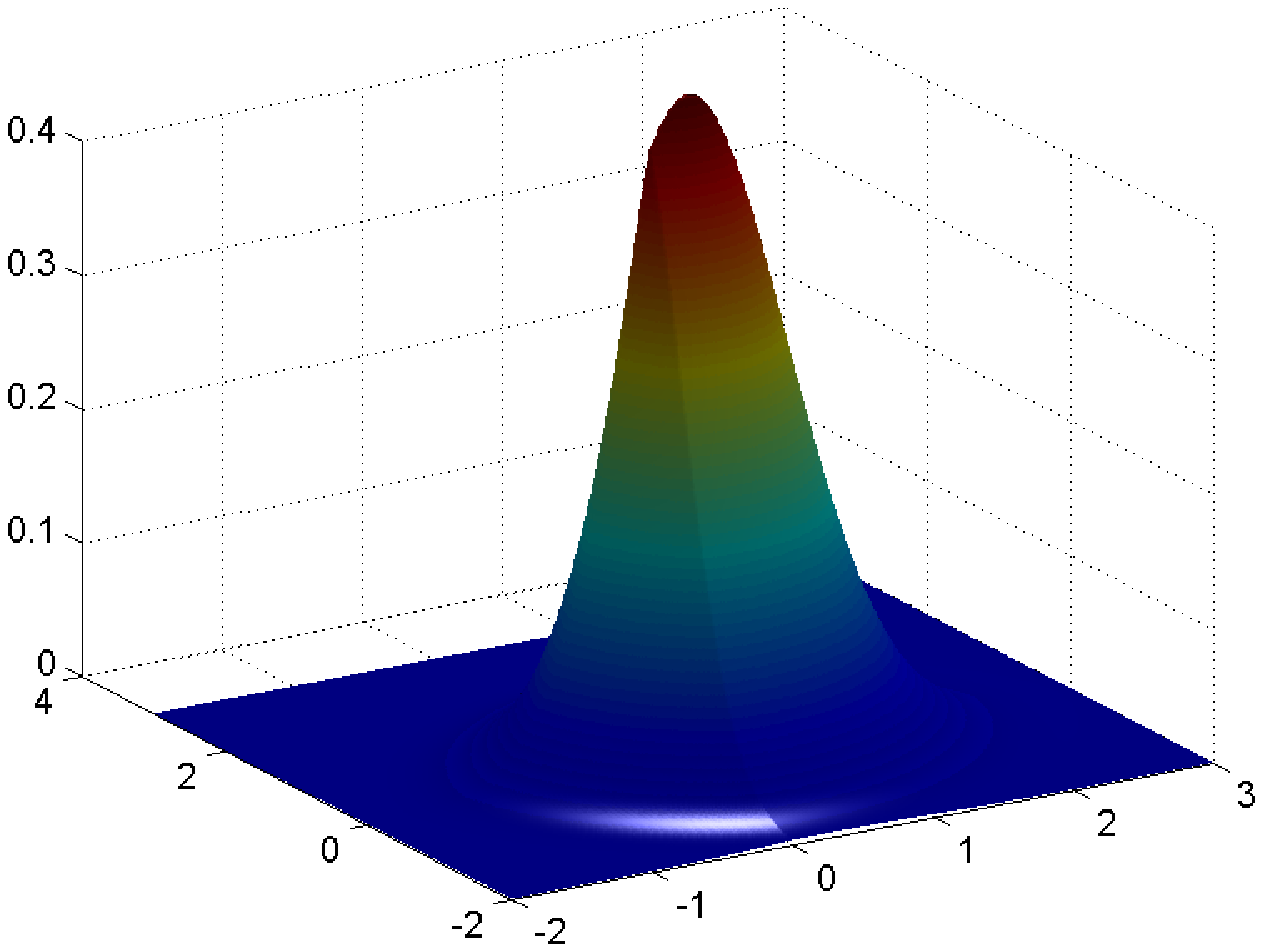}\ \ \includegraphics[width = 150pt]{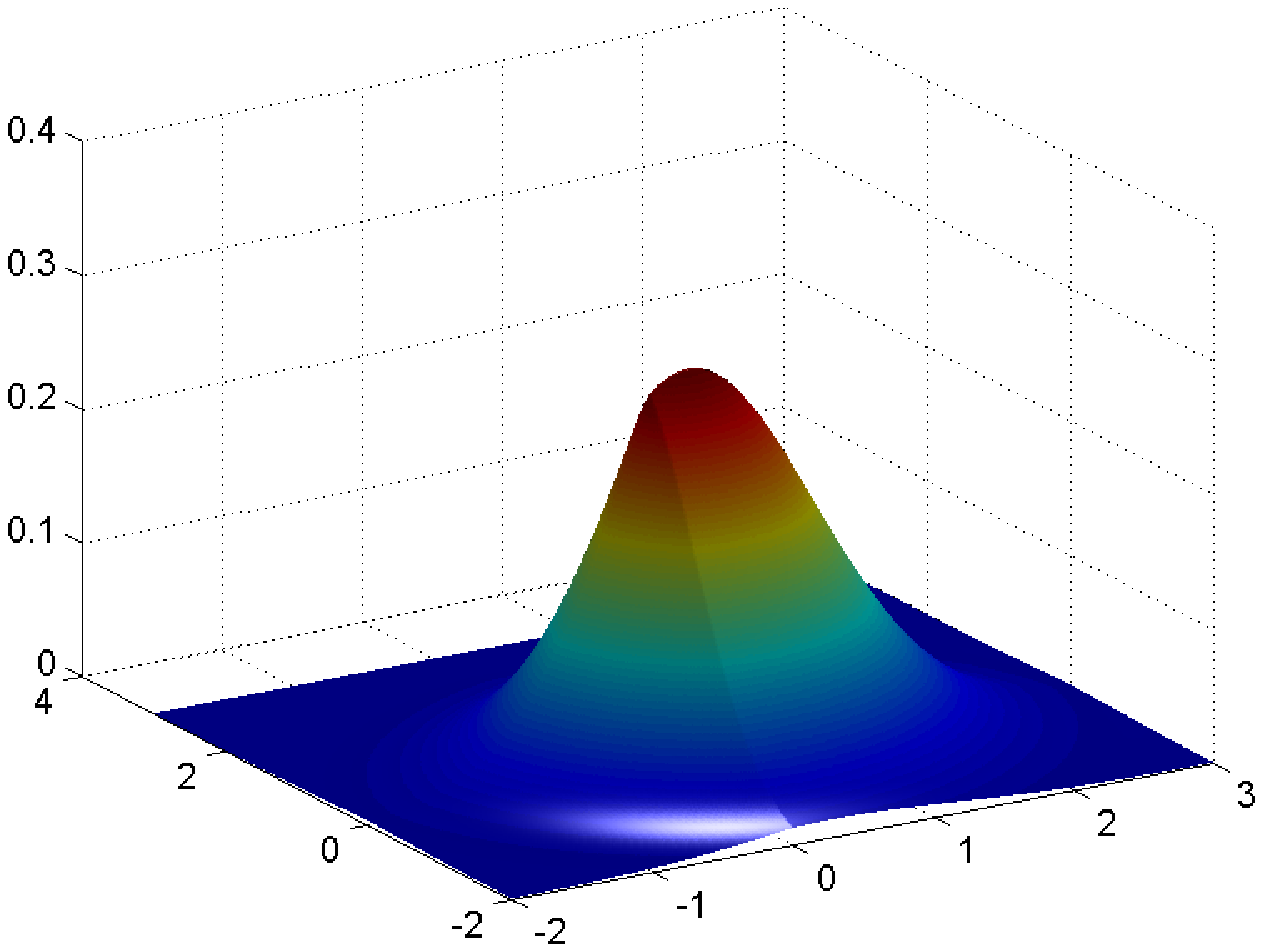}\\
	\includegraphics[width = 150pt]{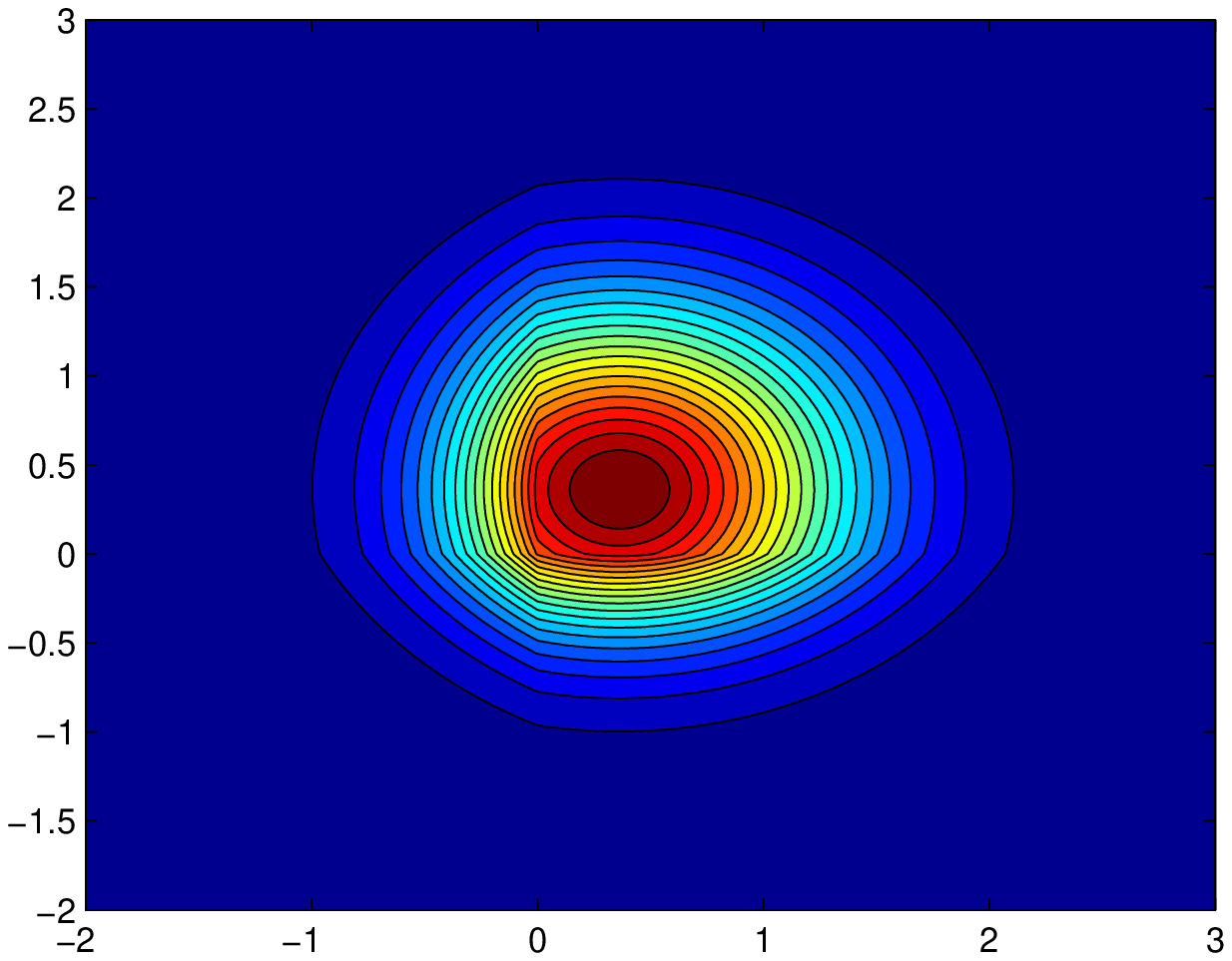}\ \ \includegraphics[width = 150pt]{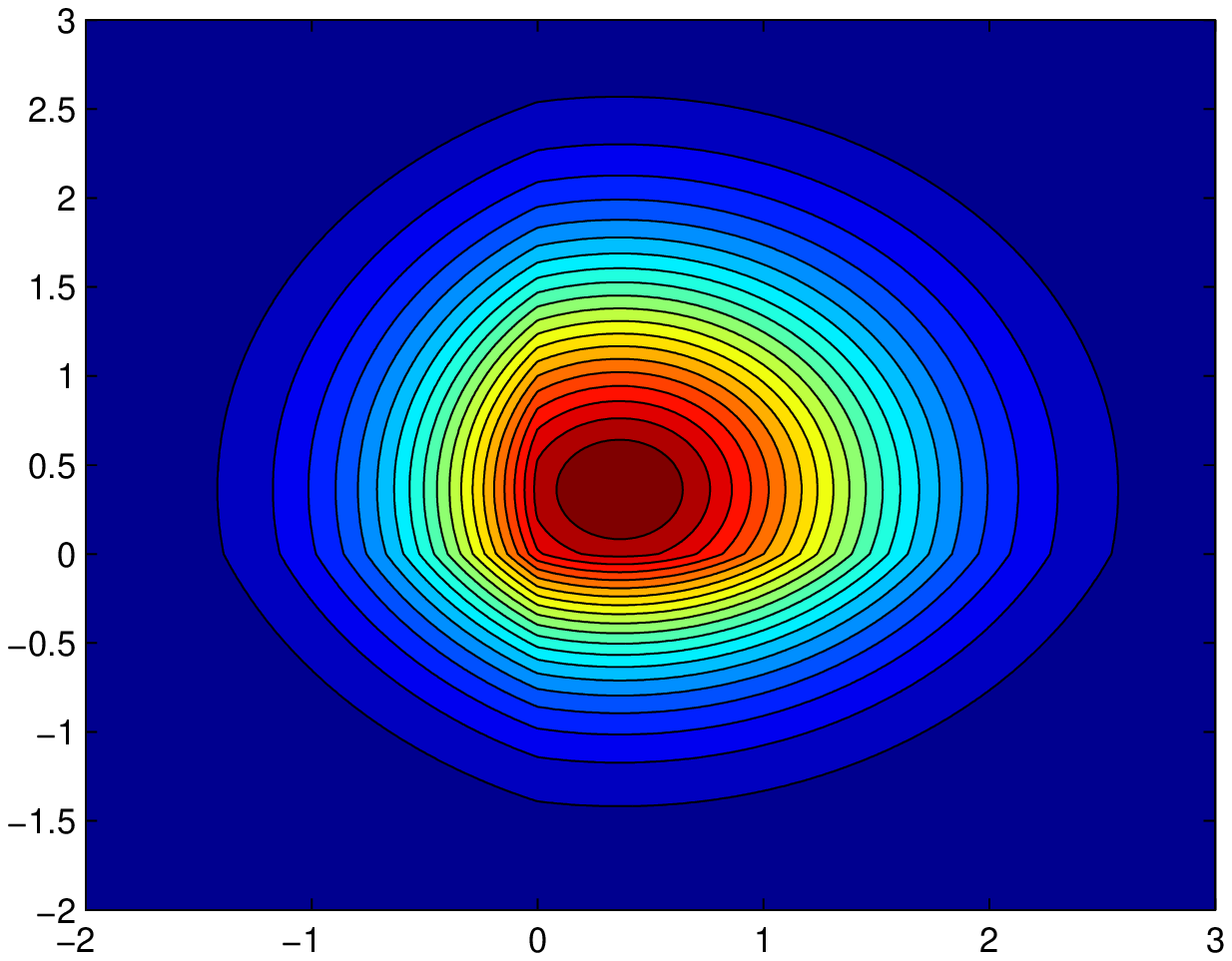}\\
    \hbox to \hsize{\hfill $\tau = 0.5$ \hfill $\tau=0.8$ \hfill}
	\caption{\textbf{Top:} the plots of the pseudo-posterior $\hat\pi_n$ with the Laplace prior for the temperature $\tau = 0.5$
    (left) and $\tau = 0.8$ (right). One can observe that decreasing the value of $\tau$ strengthens the peakedness of the density.
    \textbf{Bottom:} the level curves of the pseudo-posterior $\hat\pi_n$ with the Laplace prior for the temperature $\tau = 0.5$
    (left) and $\tau = 0.8$ (right). One clearly observes the non-differentiability of the density along the axes $\beta_1$ and $\beta_2$
    (caused by the non-differentiability of the $\ell_1$-norm).}
	\label{fig:a}
\end{figure}

Along with these studies, several authors have demonstrated the ability of the EWA to optimally estimate a sparse signal.
To this end, various types of priors have been used. For instance, \citep{Leung2006,RT11, Alquier3,arias-castro2014}
have employed discrete priors over the set of least-squares estimators with varying supports whereas \citep{DT08,DT12b}
have used Student-type heavy-tailed priors. In the context of structured sparsity, the EWA has been successfully used
in \citep{Alquier2,Alquier1,DIT}. Given the close relationship between the EWA and the Bayes estimator, it is worth
mentioning here that the problem of sparse estimation has also received much attention in the literature on Bayesian
Statistics \citep{WipfPR03,Park,Hans}. Posterior concentration properties for these methods have been investigated
in \citep{Castillo2, Castillo1, vdPas, Chao1}.

Despite these efforts, some natural questions remain open. One of them, described in details below,
is at the origin of this work. Let us consider the prediction error of a candidate
vector $\bbeta$ with respect to the quadratic loss
\begin{equation}\label{loss}
\ell_{n}(\bbeta,\bbetaz)=\frac{1}{n}\|\bfX(\bbeta-\bbetaz)\|^{2}_{2}=\frac1n\sum_{i=1}^n
(\bx_i^\top\bbeta-\bx_i^\top\bbetaz)^2.
\end{equation}
On the one hand,  theoretical studies of the lasso \citep{CandesTao,BRT,BCW2014,DHL14,BLT16,BDGP},
established\footnote{Provided that the Gram matrix $\bfX^\top\bfX/n$ satisfies suitable
assumptions (restricted isometry, restricted eigenvalues, compatibility, etc.).}
sharp upper bounds for the prediction risk of the PLS estimator \eqref{PLS}
for the $\ell_1$-penalty $\text{Pen}(\bbeta) = \|\bbeta\|_1$.

Therefore, one could expect the EWA with the Laplace prior
$\pi_0(\bbeta)\propto \exp(-\lambda\|\bbeta\|_1/\tau)$  to have a high prediction performance.
On the other hand, to the best of our knowledge, there is no result in the literature establishing accurate risk bounds for
the EWA with Laplace prior. Indeed, a straightforward application of the PAC-Bayesian type risk bounds \citep{PAC-Bayes}
for the EWA (such as, for instance, Theorem 1 in \citep{DT12b}) to the Laplace prior leads to strongly sub-optimal
remainder terms. This raises the following questions:
\vspace{-10pt}
\begin{description}
\item[Q1.] Is the EWA with the Laplace prior suitable for prediction under the sparsity scenario?
\item[Q2.] If it is, what is the range of temperature $\tau$ providing good prediction accuracy?
\item[Q3.] How do the statistical properties of the EWA with the Laplace prior compare with those of the lasso?
\end{description}

\vspace{-10pt}
Related questions are considered in \citep{Castillo1}. Indeed, for $\bbeta^\star=\mathbf 0_p$, $p=n$ and
$\bfX^\top\bfX/n =\bfI_n$,
Theorem 7 from \citep{Castillo1} establishes the following property. For all the reasonable choices\footnote{By ``reasonable''
we understand here the choice $\lambda = \text{Const}\,\sigma (\frac{\log p}{n})^{1/2}$, for which the lasso is provably
rate optimal under the sparsity scenario, provided that the design satisfies a version of the restricted eigenvalue condition.}
of the tuning parameter $\lambda$, if the temperature $\tau$ in the EWA with the Laplace prior is chosen as $\tau = \sigma^2/n$, then
the resulting posterior puts asymptotically no mass on the ball centred at $\bbetaz$ and of radius $\text{Const}({\log n}/{n})^{1/2}$,
the latter corresponding to the optimal rate of convergence in this model. This negative result, stated in terms of the posterior
contraction rate, can be easily adapted in order to show that, under the previous conditions, the Bayesian posterior mean is sub-optimal.

The present paper completes the picture by establishing some positive results. In particular, it turns out that
if the temperature parameter of the EWA with the Laplace prior is of the order $s\sigma^2/(pn)$, where $s$ is the sparsity of $\bbetaz$, then
the EWA with the Laplace prior does attain the optimal rate of convergence. Furthermore, it satisfies the same type of sharp sparsity
inequality as the lasso does. Interestingly, the proof of  this result is based on arguments which differ from those used in the aggregation
literature. Indeed, the two previously used techniques for getting oracle inequalities for the EWA and related procedures rely
either on the PAC-Bayesian inequality or on the Stein unbiased risk estimate. Instead, the key idea of our proof is to take advantage
of the following relations:
\begin{equation}
\int_{\R^p} \nabla \big(\beta_j^\alpha e^{-V_n(\bbeta)/\tau}\big)\,{\rm d}\bbeta = 0,\qquad j=1,\ldots,p,\quad \alpha = 0,1.
\end{equation}
Hence, most of our arguments are independent of the noise distribution and can be extended to other settings
(as opposed to the results relying on the Stein formula). Elaborating on this, we prove that the pseudo-posterior $\hat\pi_n$
puts an overwhelming weight on the set of vectors $\bbeta$ satisfying a sharp oracle inequality with rate-optimal remainder term.
In the case of the Gaussian noise, we also obtain the explicit form of the Stein unbiased estimator of the risk of $\bEWA$, which
can be used for choosing the tuning parameter. Finally, we extend these results to the model of trace regression when the underlying true
matrix $\bfB^\star$ has low rank.

The rest of the paper is organised as follows. The  notation used throughout the paper  is introduced in the next section.
\Cref{sec:risk}  analyses the prediction loss of the EWA with the Laplace prior,  and \Cref{sec:post}
gathers results characterising the concentration of the pseudo-posterior $\hat\pi_n$. Extensions of these results to the
case  where the unknown parameter is a (nearly) low-rank matrix are considered in \Cref{sec:matrix}. A brief summary of
the obtained results along with some conclusions is given in \Cref{sec:concl}. Finally, the proofs are postponed to \Cref{sec:proofs}.

\section{Notation}\label{sec:not}

This paragraph collects notation used throughout the paper. For every integer $k\ge 1$, we write $\mathbf 1_k$ (resp.\ $\mathbf 0_k$) for
the vector of $\R^k$ having all coordinates equal to one (resp.\ zero). We set $[k]=\{1,\ldots,k\}$. For every $q\in[0,\infty]$, we denote
by $\|\bu\|_q$ the usual $\ell_q$-norm of $\bu\in\R^k$, that is
$\|\bu\|_q =(\sum_{j\in [k]}|u_j|^q)^{1/q}$ when $0<q<\infty$, $\|\bu\|_0 = \text{Card}(\{j:u_j\not=0\})$ and $\|\bu\|_\infty =
\max_{j\in[k]} |u_j|$. For every integer $k\ge 1$ and any $T\subset [k]$, we denote by $T^c$ and
$|T|$ the complementary set $[p]\setminus T$ and the cardinality of $T$, respectively. For $\bu\in\R^k$ and $T\subset [k]$, we denote
$\bu_T\in\R^{|T|}$ the vector obtained from $\bu$ by removing all the coordinates belonging to the set $T^c$.

In Sections \ref{sec:risk} and \ref{sec:post},
we recall that $\bfX\in\R^{n\times p}$ refers to the deterministic design matrix with columns $\bx^1,\dots,\bx^p\in\R^n$ and rows
$\bx_1,\dots,\bx_n\in\R^p$. Finally, our analysis will involve the compatibility factor of the design matrix defined, for any
$J\subset [p]$ and $c>0$, by
\begin{equation}
\label{CF}
\kappa_{J,c}=\inf_{\bu\in\R^{p}:\|\bu_{J^{c}}\|_{1}<c\| \bu_{J}\|_{1}}\frac{c^{2}| J|
\| \mathbf X\bu\|^{2}_{2}}{n(c\|\bu_{J}\|_{1}-\|\bu_{J^{c}}\|_{1})^{2}}.
\end{equation}
Note that the compatibility factor, often used for the analysis of the lasso, is slightly larger\footnote{Since
this factor appears in the denominator of the risk bound, the larger is the better.} than the restricted
eigenvalue \citep{BRT}. For a better understanding of these (and related) quantities we refer the reader
to~\citep[Sections 3 and 4]{BRT} and \citep{VandeGeerConditionLasso09}.

Risk bounds established in the present work for the EWA contain a new term, as compared to the analogous
risk bounds for the lasso. This term reflects the peakedness of the pseudo-posterior density $\hat\pi_n$ and is defined by
\begin{equation}
\label{H}
H(\tau)=p\tau-\int G(\bu)\hat\pi_n(\bu){\rm d}\bu+G(\bEWA),
\end{equation}
where $G(\bu)=\nicefrac{1}{n}\|\bfX\bu\|^{2}_{2}+\lambda\|\bu\|_{1}$. When the temperature $\tau$ is low,
close to zero, the pseudo-posterior $\hat\pi_n$ is close to a Dirac measure centred at the lasso, which implies
that $H(\tau)$ is close to zero. Furthermore, since the above function $G$ is convex, we have the following
bound
\begin{equation}
H(\tau)\le p\tau.
\end{equation}
In \Cref{sec:risk} and \Cref{sec:post} we will occasionally use the following matrix notation.
For all integers $p\ge 1$, ${\bf I}_p$ refers to the identity matrix in $\R^{p\times p}$. For any integers $p\ge 1$ and $q\ge 1$, any matrix $\bfA\in \R^{p\times q}$ and any subset $T$ of $[q]$, we denote by $\bfA_T$ the matrix obtained from $\bfA$ by removing all the columns belonging
to $T^c$. Finally the transpose and the Moore-Penrose pseudoinverse of a matrix $\bfA$ are denoted by
$\bfA^\top$ and $\bfA^\dag$, respectively.


\section{Risk bound for the EWA with the Laplace prior}\label{sec:risk}

This section is devoted to discussing statistical properties of the EWA with the Laplace prior. Recall that it is defined
by \eqref{EWALP} as the average with respect to the pseudo-posterior density
\begin{equation}
\label{potential1}
\hat\pi_n(\bbeta)\propto \exp(-V_n(\bbeta)/\tau), \qquad
\text{where}\quad V_n(\bbeta)=\frac1{2n}\sum_{i=1}^n (y_i-\bx_i^\top\bbeta)^2 + \lambda\, \|\bbeta\|_1.
\end{equation}
The emphasis is put on non-asymptotic guarantees in terms of the prediction loss. It is important
to mention here that the Laplace prior, $\pi_0(\bbeta)\propto \exp(-\lambda\|\bbeta\|_1/\tau)$,
makes use of the same scale for all the coordinates of the vector $\bbeta$. This presumes that the
covariates (columns of the matrix $\bfX$) are already rescaled so that their Euclidean norms are
almost equal. An alternative approach (see, for instance, \cite{Bunea2007,BRT})---that we will not
follow here---would consist in replacing the $\ell_1$-norm of $\bbeta$ by the weighted $\ell_1$-norm
$\sum_{j\in[p]} \|\bx^j\| |\beta_j|$. The next result provides the main risk bound for the EWA.

\begin{theorem}
\label{OI1}
Assume that data are generated by model \eqref{model} with $\bxi$ drawn from the Gaussian distribution $\mathcal N(0,\sigma^2\bfI_n)$
and that the covariates are rescaled so that $\max_{j\in[p]}\nicefrac1n\|\bx^j\|_2^{2}\le 1$.
Suppose, in addition, that $\lambda\ge 2\sigma({\nicefrac2n\log(p/\delta)})^{1/2}$, for some
$\delta\in(0,1)$.  Then, with probability at least $1-\delta$,
\begin{equation}\label{SOI}
\ell_{n}(\bEWA,\betaz ) \le \inf_{\substack{\barbbeta\in\R^{p} \\ J\subset[p]}}
\bigg\{\ell_{n}(\barbbeta,\betaz)+4\lambda \|\barbbeta_{J^{c}}\|_{1}+
\frac{9\lambda^{2}| J|}{4\kappa_{J,3}}\bigg\}+2p\tau,
\end{equation}
where $\ell_{n}$ is defined in \eqref{loss} and $\bEWA$ is defined in \eqref{EWALP} and \eqref{potential1}.
\end{theorem}

For the lasso estimator, risk bounds of this nature have been developed in \citep{KLT11, SZ12, DHL14,BDGP}.
The risk bound in \eqref{SOI} extends the risk bounds available for the lasso (cf.\ Theorem 2 in \citep{DHL14})
to the EWA with the Laplace prior. Indeed, letting the temperature $\tau$ go to zero, the last term in the right-hand
side of \eqref{SOI} disappears and we retrieve the risk bound for the lasso. An attractive feature of risk bound
\eqref{SOI} is that the factor in front of the term $\ell_{n}(\barbbeta,\betaz )$ is equal to one; this is
often referred to as a sharp or exact oracle inequality. Furthermore, the other three terms in the right-hand side
of \eqref{SOI} are neat and have a simple interpretation. The second term, $4\lambda\|\barbbeta_{J^c}\|_1$, accounts
for the approximate sparsity; when $\bfX\bbetaz$ is well approximated by $\bfX\barbbeta$ with a $s$-sparse vector
$\barbbeta$, then choosing $J=\{j:\bar\beta_j\neq 0\}$ annihilates this term.
The third term of the risk bound corresponds to the optimal rate, up to a logarithmic factor, of estimation of a
vector $\bbetaz$ concentrated on the known set $J$. Indeed, if $|J| = s$ and the compatibility factor is bounded away
from zero, this term is of order $\nicefrac{s}{n}\log(p)$. Finally, the last term in the above risk bound, $2p\tau$,
reflects the influence of the temperature parameter $\tau$. In particular, it shows that if $\tau = \sigma^2/(pn)$
then this term is negligible with respect to the other remainder terms.

The inequality stated in \Cref{OI1} is a simplified version of the following one (proved in \Cref{sec:proofs}):
for any $\gamma>1$, in the event $\|\bfX^{\top}\bxi\|_\infty\leq n \lambda/\gamma$, it holds
\begin{equation}\label{SOI2}
\ell_{n}(\bEWA,\betaz ) \le \inf_{\substack{\barbbeta\in\R^{p} \\ J\subset[p]}}
\bigg\{\ell_{n}(\barbbeta,\betaz )+4\lambda \|\barbbeta_{J^{c}}\|_{1}+
\frac{\lambda^{2}(\gamma+1)^{2}| J|}{\gamma^{2}\kappa_{J,(\gamma+1)/(\gamma-1)}}\bigg\}+2H(\tau),
\end{equation}
where $H(\tau)$ is defined in \eqref{H}. On the one hand, one can use this more general result for
getting an oracle inequality under more general assumptions on the noise distribution such as
those considered, for instance, in \citep{Bunea2007,BCW2014}. On the other hand, one can infer
from \eqref{SOI2} that the term $H(\tau)$ highlights the difference, in terms of statistical
complexity, between the lasso and the EWA with the Laplace prior. It is therefore important to
get a precise evaluation of $H(\tau)$ as a function of $\tau$, $p$ and $n$, and to understand how
tight the inequality $H(\tau)\le p\tau$ is. To answer this question, we restrict our attention
to orthonormal designs and show the tightness of the aforementioned inequality. To this end,
let us introduce the scaled complementary error function $\Psi_v(t) =
e^{t^2/2v}\frac1{\sqrt{2\pi v}}\int_t^\infty e^{-u^2/2v}\,{\rm d}u$.

\begin{prop}
\label{prop:evalH}
Let $\hat\bfSigma_n = \nicefrac1n\bfX^\top\bfX$ be the Gram matrix and $\bbetaLS = \nicefrac1n\hat\bfSigma_n^{\dag}\bfX^\top\by$
be the least-squares estimator. Then, we have
\begin{align}
H(\tau)
    = \|\hat\bfSigma_n^{1/2}\bEWA\|_2^2+\lambda\|\bEWA\|_1-(\bEWA)^\top\hat\bfSigma_n\bbetaLS.
\end{align}
Furthermore, when the design is orthonormal, that is $\hat\bfSigma_n = \bfI_p$, then the EWA with the
Laplace prior is a thresholding estimator,
$\EWA_j= \sgn(\betaLS_j)\big(|\betaLS_j| -\lambda w(\tau,\lambda,|\betaLS_j|)\big)$, where
\begin{align}
w(\tau,\lambda,|\betaLS_j|)
    = \frac{\Psi_\tau(\lambda-|\betaLS_j|)-\Psi_\tau(\lambda+|\betaLS_j|)}
		{\Psi_\tau(\lambda-|\betaLS_j|)+\Psi_\tau(\lambda+|\betaLS_j|)},
\end{align}
and
\begin{equation}\label{H1}
H(\tau) = \sum_{j=1}^p \lambda \big(|\betaLS_j| -\lambda w(\tau,\lambda,|\betaLS_j|)\big)\big(1-w(\tau,\lambda,|\betaLS_j|)\big).
\end{equation}
\end{prop}
The last expression of $H(\tau)$ provided by the proposition may be used for a numerical evaluation.
First, let us note that if we set $\bar\beta_j = \betaLS_j/\sqrt{\tau}$ and $\bar\lambda = \lambda/\sqrt{\tau}$,
the function $H(\tau)/\tau$ is independent of $\tau$. Indeed, we have $H(\tau)/\tau = \sum_j h(\bar\lambda,|\bar\beta_j|)$ where
$$
h(\bar\lambda, z) = \bar\lambda \big(z -\bar\lambda w(1,\bar\lambda,z)\big)\big(1-w(1,\bar\lambda,z)\big),\qquad \forall z>0.
$$
In Fig.~\ref{fig:1} below, we plot the curves of the functions $z\mapsto h(\bar\lambda,z)$ for different values of the parameter
$\bar\lambda$.
\begin{figure}[ht]
\includegraphics[height = 150pt]{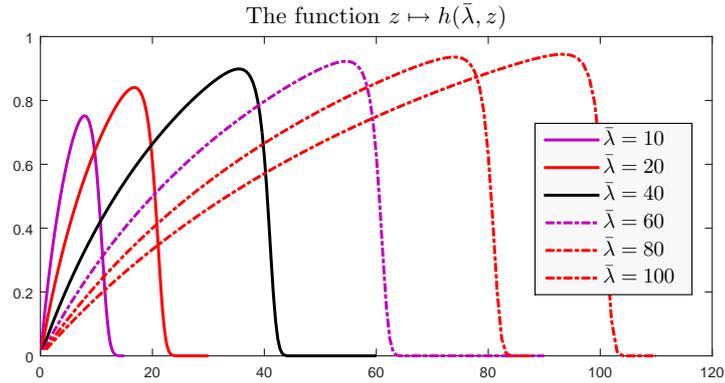}	
	\label{fig:1}
	\caption{For different values $\bar\lambda\in \{10,20,40,60,80,100\}$, we plot the
	function $z\mapsto h(\bar\lambda,z)$.}
\end{figure}

\noindent These curves clearly show that the bound $H(\tau)\le p\tau$, a consequence of $h(\bar\lambda,z)\le 1$, is tight.
Another interesting observation is that the function $H(\tau)$ is always nonnegative. This basically implies that the value
of $\tau$ minimising the right-hand side of \eqref{SOI2} is $\tau = 0$. In other terms, the lowest risk bound is obtained
for the lasso. This legitimately raises the following question: is there any advantage of using the EWA with the Laplace
prior as compared to the lasso? Our firm conviction is that there is an advantage, and will try to explain our viewpoint in the
rest of this section.

The point is that the lasso estimator is a nonsmooth function of the data. One of the consequences of this is that the
Stein unbiased risk estimate (SURE) for the lasso is a discontinuous function of data. Indeed, as proved in \citep{Tibsh12},
The SURE for the lasso (see also the earlier work \citep{Donoho95,Zou2007}) is given by
$$
\widehat R^{\rm lasso} (\lambda) = \frac1n \|\by - \bfX\hatbbeta{}^{\rm lasso}(\lambda)\|_2^2 - \sigma^2+\frac{2\sigma^2}{n}
\text{rank}(\bfX_{\mca(\lambda)}),
$$
where $\mca(\lambda) = \{j\in[p]: \hat\beta{}_j^{\rm lasso}(\lambda)\neq 0\}$ is the active set for the lasso
estimator with the tuning parameter $\lambda$. In theory, this quantity $\widehat R^{\rm lasso}(\lambda)$ can be used
for choosing the tuning parameter $\lambda$ of the lasso. However, in practice, this solution is rarely employed,
since $\mca(\lambda)$ has a very unstable behaviour as a function of $\lambda$ and $\by$. As a consequence,
not only one can get very different ``optimal'' values of $\lambda$ for two very close vectors $\by$ and $\by'$,
but is also likely to obtain  very different ``optimal'' values of $\lambda$ for the same  vector $\by$  if using
two different optimisation algorithms for computing an approximate solution to the lasso problem.

Using Stein's lemma, in the case where $\bxi$ is drawn from the Gaussian $\mathcal N(0,\sigma^2\bfI_n)$
distribution, one checks that
\begin{equation}
\hat R^{\rm EWA}(\lambda,\tau) =
\frac1n\| \by - \bfX \bEWA_{\lambda,\tau} \|_2^{2} - \frac{\sigma^2}{n}  + \frac{2\sigma^2}{n^2\tau} \int_{\RR^p}
\|\bfX(\bbeta-\bEWA_{\lambda,\tau})\|_2^2\,\hat\pi_{n,\lambda,\tau}(\bbeta)\,{\rm d}\bbeta
\label{riskEstim}
\end{equation}
is an unbiased estimator of the risk $\esp[\ell_n(\bEWA,\bbetaz)]$. Furthermore,  the function
$(\lambda,\tau)\mapsto \hat R^{\rm EWA}(\lambda,\tau)$ is clearly continuous on $(0,\infty)\times(0,\infty)$.
One can also check that the unbiased risk estimate $\hat R^{\rm EWA}(\lambda,\tau)$ depends continuously  on the data
vector $\by$. Therefore, this quantity us arguably  more robust to the variation in data and  more regular as a
function of the tuning parameters as compared to $\hat R^{\rm lasso}$. This implies that minimising
$\hat R^{\rm EWA}(\lambda,\tau)$ with respect to $\lambda$ or $\tau$ might be a good strategy for choosing
these parameters adaptively.

Of course, this requires to be able to numerically compute the right-hand side of \eqref{riskEstim} or,
equivalently, the mean and the covariance matrix of the pseudo-posterior distribution $\hat\pi_n$. For
smooth and strongly log-concave densities, the cost of such computations has been recently
assessed in \citep{Dal14,Durmus16}. The adaptation of the approaches developed therein to the pseudo-posterior
$\hat\pi_n$, which is neither smooth nor strongly log-concave (but can be approximated by such a function),
is an ongoing work.


\section{Pseudo-Posterior concentration}
\label{sec:post}

Since the EWA estimator has a Bayesian flavour, it is appealing to look at the concentration properties
of the pseudo-posterior distribution $\hat\pi_n$. This is particularly important in the light of the results
in~\cite{Castillo1} establishing that, for the temperature $\tau = \sigma^2/n$, the pseudo-posterior $\hat\pi_n$
with the Laplace prior puts asymptotically no mass on the set of vectors $\bbeta$ having a small prediction
error. Furthermore, this result is proven for the orthonormal design matrix $\bfX$, which, intuitively, is a
rather favourable situation for the Laplace prior.

The first property that we establish here and that characterises the concentration of the pseudo-posterior
around its average is the following upper bound on the variance of the prediction $\bfX\bbeta$ when $\bbeta$ is
drawn from $\hat\pi_n$. (Recall that the matrix $\bfX$ has $n$ rows, so the normalisation by multiplicative
factor $1/n$ is natural.)

\begin{prop}
\label{variancePosteriorBound}
If $\hat\pi_n(\bu)\propto \exp{(-V_n(\bu)/\tau)}$ is the pseudo-posterior with the Laplace prior defined
by \eqref{potential1}, then, for every $\barbbeta\in\R^p$,
we have
\begin{equation}\label{lemConcentration}
\int_{\R^p} V_n(\bu)\, \hat\pi_n(\bu)\, {\rm d}\bu \leq p \tau + V_n(\barbbeta) -
\frac{1}{2n} \int_{\mathbb{R}^{p}} \| \bfX (\bu - \barbbeta) \|_{2}^{2}\,\hat\pi_n(\bu)\, {\rm d}\bu.
\end{equation}
Furthermore, choosing $\barbbeta = \bEWA = \int_{\R^p} \bu\,\hat\pi_n(\bu)\,{\rm d}\bu$, we get
\begin{equation}
\label{varianceInequality}
\frac{1}{n} \int_{\R^p}\| \bfX (\bu - \bEWA ) \|_{2}^{2}\,\hat\pi_n(\bu)\, {\rm d}\bu \leq p\tau.
\end{equation}
\end{prop}
The proof of this result is rather simple and plays an important role in the proof of the oracle inequality stated in \Cref{OI1}.
For these reasons, we opted for presenting this proof in this section, instead of postponing it to \Cref{sec:proofs}.

\begin{proof} The convexity of the function $\barbbeta \mapsto \| \barbbeta \|_1$ readily implies
that the function $\barbbeta\mapsto W_n(\barbbeta) = V_n(\barbbeta)- \nicefrac{1}{2n} \| \bfX(\bu - \barbbeta) \|_{2}^{2}$ is a convex
function, for every fixed $\bu\in\R^p$. Furthermore, we have $W_n(\bu) = V_n(\bu)$ and $\nabla W_n(\bu) = \nabla V_n(\bu)$ at
any point $\bu$ of differentiability of $V_n$. Therefore,
\begin{equation}\label{eq:1}
V_n \big(\barbbeta \big) \geq V_n(\bu) + \big(\barbbeta - \bu \big)^{\top}\nabla V_n(\bu) + \frac{1}{2n} \big\| \bfX(\bu-\barbbeta) \big\|_{2}^{2},
\end{equation}
for all $\barbbeta \in \R^p$ and for almost all $\bu\in\R^p$ (those for which $V_n$ is continuously differentiable at $\bu$).
Using the fundamental theorem of calculus, we remark that
\begin{equation}
\label{lemConcentrationRk1}
\int_{\mathbb{R}^{p}} \nabla V_n(\bu)\, \hat\pi(\bu)\, {\rm d}\bu  = -\tau \int_{\mathbb{R}^{p}}[\nabla \hat\pi_n(\bu)]\, {\rm d}\bu = \mathbf 0_p
\end{equation}
and that
\begin{align}
\int_{\R^{p}}
\bu^{\top} \nabla V_n(\bu)\, \hat\pi(\bu)\, {\rm d}\bu - p\tau &= \int_{\R^{p}} \sum_{j=1}^{p} \bigg( \beta_j \frac{\partial V_n}{\partial \beta_j}(\bu) -\tau\bigg)  \,\hat\pi(\bu)\, {\rm d}\bu \\
&= -\tau \int_{\R^{p}} \sum_{j=1}^{p} \frac{\partial[ u_j \hat\pi_n(\bu)]}{\partial u_j}\, {\rm d}\bu = 0. \label{lemConcentrationRk2}
\end{align}
Integrating inequality \eqref{eq:1} on $\R^p$ with respect to the density $\hat\pi_n$ and using relations \eqref{lemConcentrationRk1} and \eqref{lemConcentrationRk2}, we arrive at
\begin{equation}\label{eq:2}
V_n \big(\barbbeta \big) \geq \int_{\R^{p}} V_n(\bu)\, \hat\pi_n(\bu)\, {\rm d}\bu - p\tau + \frac{1}{2n} \int_{\R^{p}}
\big\|  \bfX \big(\bu - \barbbeta \big) \big\|_{2}^{2}\,\hat\pi_n(\bu)\, {\rm d}\bu.
\end{equation}
This completes the proof of the first claim of the proposition.

To prove the second claim, we replace $\barbbeta$ by $\bEWA$ in \eqref{eq:2}. After rearranging the terms, this yields
\begin{equation}\label{eq:3}
\frac{1}{2n} \int_{\R^{p}} \big\|  \bfX \big(\bu -\bEWA \big) \big\|_{2}^{2}\,\hat\pi_n(\bu)\, {\rm d}\bu \le
p\tau + V_n \big(\bEWA\big) -\int_{\R^{p}} V_n(\bu)\, \hat\pi_n(\bu)\, {\rm d}\bu.
\end{equation}
Using once again the fact that $\bu\mapsto W_n(\bu) = V_n(\bu)- \nicefrac{1}{2n} \| \bfX(\bu - \bEWA) \|_{2}^{2}$ is a convex function,
we obtain $V_n \big(\bEWA\big) = W_n \big(\bEWA\big) \le \int W_n(\bu)\,\hat\pi_n(\bu)\,{\rm d}\bu$, which is equivalent to
\begin{align}
V_n \big(\bEWA\big) -\int_{\R^{p}} V_n(\bu)\, \hat\pi_n(\bu)\,{\rm d}\bu
&\le - \frac{1}{2n} \int_{\R^{p}} \big\|  \bfX \big(\bu -\bEWA \big) \big\|_{2}^{2}\,\hat\pi_n(\bu)\,{\rm d}\bu.\label{eq:4}
\end{align}
This inequality, combined with \eqref{eq:3}, completes the proof of \eqref{varianceInequality} and of the proposition.
\end{proof}

\begin{rem}
A careful inspection of the proof reveals that the claims of the proposition are independent of the precise form of the $\ell_1$-penalty.
Therefore, the proposition still holds if we replace the $\ell_1$-norm by any convex penalty.
\end{rem}

The second claim of the proposition establishes that the dispersion of the distribution $\hat\pi_n$ around its average value
$\bEWA$ is of the order $(p\tau)^{1/2}$. Interestingly, we show below that the same order of magnitude appears when we determine a
region of concentration for the pseudo-posterior $\hat\pi_n$. A key argument in the proof of the latter claim is the following result.

\begin{prop}[\cite{Bobkov11}, Theorem~1.1]
\label{propBobkov}
Let $\hat\pi_n(\bu)\propto\exp(-V_n(\bu)/\tau)$ be a log-concave probability density\/\footnote{This means that $V_n$ is a
convex function.} and let $\bbeta$ be a random vector drawn from $\hat\pi_n$. Then, for any $t>0$, the inequality
\begin{equation}
\label{eqBobkov}
 V_n(\bbeta) \leq  \int_{\R^{p}}V_n(\bu)\,\hat\pi_n(\bu)\, {\rm d}\bu +  \tau \sqrt{p}\, t
\end{equation}
holds with probability at least $1-2e^{-t/16}$.
\end{prop}

Using this proposition, we establish the following result (the proof of which is postponed to \Cref{sec:proofs})
characterising the concentration of $\hat\pi_n$.

\begin{theorem}[Posterior concentration bound]
\label{theoPosteriorConcentration}
Assume that data are generated by model \eqref{model} with  $\bxi\sim\mathcal N(0,\sigma^2\bfI_n)$
and rescaled covariates, \textit{i.e.}, $\max_{j\in[p]}\nicefrac1n\|\bx^j\|_2^{2}\le 1$. Let the quality of an
estimator be measured by the squared prediction loss \eqref{loss}. Assume that the tuning parameter
$\lambda$ satisfies $\lambda\ge 2\sigma \big(\nicefrac{2}{n}\log(\nicefrac{p}{\delta})\big){}^{1/2}$, for some
$\delta\in(0,1)$.  Then, with probability at least $1-\delta$, the pseudo-posterior $\hat\pi_n$ with the Laplace
prior defined by {\rm (\ref{potential1})} satisfies
\begin{equation}\label{eqThmConcentrationBound}
\hat\pi_n\bigg(\bbeta:
\ell_{n}(\bbeta,\betaz ) \le \inf_{\substack{\barbbeta\in\R^{p} \\ J\subset[p]}}
\bigg\{\ell_{n}(\barbbeta,\betaz )+4\lambda \|\barbbeta_{J^{c}}\|_{1}+
\frac{9\lambda^{2}| J|}{2\kappa_{J,3}}\bigg\} + 8p\tau \bigg)\ge 1-
2e^{-\sqrt{p}/16}.
\end{equation}
\end{theorem}

The latter theorem, in conjunction with \Cref{OI1}, tells us that if we generate a random vector $\bbeta$
distributed according to the density $\hat\pi_n$, then with high probability it will have a prediction loss
almost as small as the one of the EWA, the average with respect to $\hat\pi_n$.
This remark might be attractive from the computational point of view, since, at least for some distributions,
drawing a random sample is easier than computing the expectation. Note also that by increasing the factor in
front of the term $p\tau$ it is possible to make the $\hat\pi_n$-probability of the event considered
in \Cref{theoPosteriorConcentration} even closer to one.

\section{Sparsity oracle inequality in the matrix case}\label{sec:matrix}

In this section, we extend the results of the previous sections to the problem of matrix regression with a
low-rankness inducing prior. We first need to introduce additional notations used throughout this section.

\subsection{Specific notation}
For two matrices $\bfA$ and $\bfB$ of the same dimension, the scalar product
is defined by
\begin{equation}
\langle \bfA, \bfB \rangle = \tr(\bfA\!^{\top}\,\bfB).
\end{equation}
The nuclear norm of a $p\times q$ matrix $\bfA$ is $\| \bfA \|_1 = \sum_{k=1}^{r} s_{\bfA,k}$, where $s_{\bfA,k}$
is the $k$-th largest singular value of $\bfA$ and $r = \rk(\bfA)$. The operator norm is
$\| \bfA \| = \sup_{\bx \in \R^{q}} {\| \bfA \bx \|_2}/{\|\bx \|_2} = s_{\bfA,1}$.
We denote by $\bbfX = (\bfX_1, \ldots, \bfX_n) \in \R^{n \times m_1 \times m_2}$ the three-dimen\-sio\-nal tensor playing
the role of the design matrix. Besides, let $\|\bfA \|_{L_2(\bbfX)}^{2} =  \langle \bfA, \bfA \rangle_{L_2(\bbfX)}$ be the
prediction loss defined via the ``scalar product'' $\langle \bfA, \bfC \rangle_{L_2(\bbfX)} = \frac{1}{n} \sum_{i=1}^{n}
(\langle \bfX_i, \bfA \rangle\langle \bfX_i, \bfC \rangle)$. We will use the notation
$\bu\!^\top\bbfX  = \sum_{i\in[n]} u_i\bfX_i\in \bbfM$ for the product of the tensor $\bbfX$ with the
vector $\bu \in \R^n$.

We now need to define the matrix compatibility factor. Its definition is more involved than in the vector case
because of the fact that the left and right singular spaces differ from one matrix to another. Let $\barbfB$ be any
$m_1\times m_2$ matrix of rank $r=\rk(\barbfB)$ having the singular value decomposition $\barbfB = \bfV_1\bfSigma\bfV_2^\top$.
Here, $\bfSigma$ is a $r\times r$ diagonal matrix with positive diagonal entries, $\bfSigma_{11}\ge\ldots\ge \bfSigma_{rr}>0$,
and $\bfV_j$ is a $m_j\times r$ matrix with orthonormal columns for $j=1,2$. For any  $J\subset [r]$ and $j=1,2$, we define
$\bfV_{j,J}$ as the $m_j\times |J|$ matrix obtained from $\bfV_j$ by removing the columns with indices lying outside
of $J$. This allows us to introduce the linear operators $\calP_{\barbfB,J^c}$ and $\calP^\bot_{\barbfB,J^c}$ from $\bbfM$
to $\bbfM$
$$
\calP_{\barbfB,J^c}(\bfU) = (\bfI_{m_1}-\bfV_{1,J}\bfV_{1,J}^\top)\bfU(\bfI_{m_2}-\bfV_{2,J}\bfV_{2,J}^\top),\quad
\calP_{\barbfB,J^c}^\bot(\bfU) = \bfU-\calP_{\barbfB,J^c}(\bfU).
$$
We define, for every $\barbfB\in\bbfM$, $J\subset [\rk(\barbfB)]$ and $c>0$, the compatibility factor
\begin{equation}
\label{CFMatrix}
\kappa_{\barbfB,J,c} = \inf_{\substack{\bfU \in \bbfM\\ \|\calP_{\barbfB,J^c}(\bfU) \|_{1}<c\|\calP_{\barbfB,J^c}^\bot(\bfU) \|_{1}}}
\frac{c^{2}|J|\,\| \bfU\|_{L_2(\bbfX)}^{2}}{\big(c\|\calP_{\barbfB,J^c}^\bot(\bfU) \|_{1}-\|\calP_{\barbfB,J^c}(\bfU) \|_{1}\big)^{2}}.
\end{equation}
When $J = [\rk(\barbfB)]$, we use the notation  $\kappa_{\barbfB,c}$ instead of $\kappa_{\barbfB,J,c}$. Note that the set $\calC(\barbfB,J,c) =
\{\bfU \in \bbfM:\ \|\calP_{\barbfB,J^c}(\bfU)\|_{1}<c\|\calP_{\barbfB,J^c}^\bot(\bfU) \|_{1}\}$ defines the cone of dimensionality
reduction. It consists of matrices $\bfU$ that can be written as a sum of two matrices $\bfU_1$ and $\bfU_2$ such that $\bfU_1$ is of small rank
and dominates the possibly full-rank matrix $\bfU_2$, in the sense that $\|\bfU_2\|_1 \le c\|\bfU_1\|_1$. Indeed, it suffices to set
$\bfU_1 = \calP_{\barbfB,J^c}^\bot(\bfU)$ and to remark that $\calP_{\barbfB,J^c}^\bot(\bfU) = \bfV_{1,J}\bfV_{1,J}^\top\bfU +
(\bfI_{m_1}-\bfV_{1,J}\bfV_{1,J}^\top)\bfU\bfV_{2,J}\bfV_{2,J}^\top$ is of rank not exceeding $2|J|$.

Similarly to \eqref{H}, we also define the function
\begin{equation}
\label{matrixH}
H(\tau)=m_1 m_2\tau- \int_{\bbfM} G(\bfU)\,\hat\pi_n(\bfU)\,{\rm d}\bfU+G\big(\hatbfB\big),
\end{equation}
where $G(\bfU)=\| \bfU\|_{L_2(\bbfX)}^{2}+\lambda\|\bfU\|_{1}$. The convexity property of the function $G$ entails that
$H(\tau)\le m_1 m_2\tau$ for every $\tau>0$.

\subsection{Nuclear-norm prior and the exponential weights}

The observed outcomes are $n$ real random variables $y_{1},\dots,y_{n}\in\R$.
Contrary to Sections \ref{sec:risk} and \ref{sec:post} where the design points are $\bx_1,\dots,\bx_n\in\R^p$,
this section studies the situation in which we consider $n$ design matrices $\bfX_i \in \R^{m_1\times m_2}$ for $i \in [n]$.
We further assume that there is a regression
matrix $\bfBz \in\bbfM$ such that
\begin{equation}
\label{modelMatrix}
y_i=\tr(\bfX_i^{\top} \bfBz)+\xi_i,\qquad i\in[n],
\end{equation}
where the residuals $\xi_{i}$ are independent and identically distributed according to a centred Gaussian distribution
with variance $\sigma^{2}$. This model is referred to as trace-regression; see, for instance, \cite{Rohde11}.
In this model, the nuclear norm is akin to the $\ell_1$ norm in the vector case. Therefore, to some extent,
the equivalent of the lasso estimator $\hat\bfB^{\rm NNP-LS}_{\lambda}$ with a positive smoothing parameter $\lambda$,
is defined by
\begin{equation}
\hat\bfB^{\rm NNP-LS}_{\lambda}\in{\arg\min}_{\bfB \in \bbfM}
\bigg\{\frac{1}{2n} \sum_{i \in [n]}(\by_i - \langle \bfX_i, \bfB \rangle)^{2}+\lambda\|\bfB\|_{1}\bigg\}.
\end{equation}
This is the nuclear-norm penalized least-squares estimator. Similarly to the vector case, the
above defined estimator $\hat\bfB^{\rm NNP-LS}_{\lambda}$ is the maximum a posteriori estimator
corresponding to the nuclear-norm prior
\begin{equation}
\pi_{0}(\bfB)\propto\exp\Big\{-\frac{\lambda\sigma^{2}\| \bfB\|_{1}}{n}\Big\}.
\end{equation}

This section investigates the prediction performance of the procedure obtained by replacing the optimisation
step by averaging. In the matrix case, we define the potential function $V_n$ and the pseudo-posterior,
respectively, by
\begin{equation}
\label{potentialMatrix}
V_n(\bfB)=\frac{1}{2n} \sum_{i \in [n]}(\by_i - \langle \bfX_i, \bfB \rangle)^{2}+\lambda\|\bfB\|_{1},
\qquad\text{and}\qquad
\hat\pi_n(\bfB)\propto\exp\left\{-\nicefrac1\tau{V_n(\bfB)}\right\}.
\end{equation}
Using these concepts, we define the EWA with the nuclear-norm prior by
\begin{equation}
\label{EWALPMatrix}
\hatbfB^{\rm EWA} =\int_{\bbfM}\bfB\,\hat\pi_n(\bfB)\,{\rm d}\bfB.
\end{equation}
We aim at studying the performance of this estimator in terms of the prediction loss
\begin{equation}
\label{mloss}
\ell_{n}\big(\hatbfB,\bfBz\big)
    =\|\hatbfB-\bfBz\|^2_{L^2(\bbfX)}
    = \frac{1}{n}\sum_{i=1}^{n} \langle \bfX_i, \hatbfB-\bfBz\rangle^{2}.
\end{equation}

\subsection{Oracle Inequality}
The problem of assessing the quality of the nuclear-norm penalised estimators has received a great deal of
attention; see,  for instance, \citep{Srebro2005,CT10, CP11, BYW,GL11,NW11,NW12,Klopp14}. Such an interest in these methods
is mainly motivated by the variety of applications in computer vision and image analysis \citep{Shen, Harchaoui},
recommendation systems \citep{Zhou2008,Lim}, and many other areas. Bayesian approaches to the problem of low-rank matrix
estimation and prediction has been recently analysed by \cite{Alquier2013,Mai15,Cottet16}.

Making the parallel with the sparse vector estimation and prediction problem, we can note that the counterpart of the
vector sparsity $s = \|\bbetaz\|_0$ in the matrix case is the product $(m_1+m_2)\rk(\bfBz)$, representing the number of
potentially nonzero terms in the singular values decomposition of $\bfBz$. Similarly, the counterpart
of the ambient dimension $p$ is the overall number of entries in $\bfBz$ that is $m_1m_2$. In view of these
analogies, the next theorem is a natural extension of \Cref{OI1} to the model of trace-regression. To state it, we need
the following notation:
\begin{equation}\label{vX}
v_\bbfX = \bigg\|\frac1n\sum_{i=1}^n \bfX_i\bfX_i^\top\bigg\|^{1/2} \bigvee \bigg\|\frac1n\sum_{i=1}^n \bfX_i^\top\bfX_i\bigg\|^{1/2}.
\end{equation}

\begin{theorem}
\label{OI1Matrix}
Assume that data are generated by model \eqref{modelMatrix} with $\bxi$ drawn from the Gaussian distribution
$\mathcal N(\mathbf 0_n,\sigma^2\bfI_n)$.  Suppose, in addition, that
$\lambda\ge 2\sigma v_\bbfX\{{\nicefrac{2}{n}\log((m_1 + m_2)/\delta)}\}^{1/2}$, for some $\delta\in(0,1)$.
Then, with probability at least $1-\delta$, the matrix $\hatbfB^{\rm EWA}$ defined in \eqref{EWALPMatrix} satisfies
\begin{equation}\label{matrix:OI}
\ell_{n}(\hatbfB^{\rm EWA},\bfBz) \le \inf_{\substack{\barbfB\in\bbfM \\ J\subset [\rk(\barbfB)]}}
\bigg\{\ell_{n}(\barbfB,\bfBz)+4\lambda\| \calP_{\barbfB,J^c}(\barbfB) \|_{1}+
\frac{ 9 \lambda^{2} |J|}{4\kappa_{\barbfB,J,3}}\bigg\}+2m_1m_2\tau.
\end{equation}
\end{theorem}

This result can be seen as an extension of \citep[Theorem 2]{KLT11} to the exponentially weighted aggregate with a prior
proportional to the scaled nuclear norm. Indeed, if we upper bound the infimum over all matrices $\bfB$ by the infimum over
matrices such that $\rk(\bfB)\le r$ for some given integer $r$, we easily see that \eqref{matrix:OI} yields
\begin{equation}
\ell_{n}(\hatbfB^{\rm EWA},\bfBz) \le \inf_{\substack{\barbfB\in\bbfM \\ \rk(\barbfB)\le r}}
\bigg\{\ell_{n}(\barbfB,\bfBz)+\frac{ 9 \lambda^{2} r}{4\kappa_{\barbfB,3}}\bigg\}+2m_1m_2\tau.
\end{equation}
An advantage of inequality \eqref{matrix:OI} is that it offers a continuous interpolation between
the so called ``slow'' and ``fast'' rates. ``Slow'' rates refer typically to risk bounds that are proportional
to $\lambda$, whereas ``fast'' rates are proportional to $\lambda^2$. For procedures based on $\ell_1$-norm or
nuclear-norm penalty, ``slow'' rates are known to hold without any assumption on the design, while ``fast''
rates require a kind of compatibility assumption. In  \eqref{matrix:OI}, taking $J=\varnothing$, the term with
$\lambda^2$ disappears and we get the ``slow'' rate proportional to $\lambda\|\barbfB\|_1$. The other extreme
case corresponding to $J=[\rk(\barbfB)]$ leads to the ``fast'' rate proportional to $\lambda^2\rk(\barbfB)$,
provided that the compatibility factor is bounded away from zero. The risk bound in \eqref{matrix:OI} bridges
these two extreme situations by providing the rate $\min_{q\in [r]}\{\lambda(s_{q+1,\barbfB}+
\ldots+s_{r,\barbfB})+\lambda^2q\}$, where $r=\rk(\barbfB)$ and $s_{\ell,\barbfB}$ is the $\ell$-th largest
singular value of $\barbfB$. Thus, our risk bound quantifies the quality of prediction in the situations where
the true matrix (or the best prediction matrix) is nearly low-rank, but not necessarily exactly low-rank.

As well as in the vector case, the inequality stated in \Cref{OI1Matrix} is a simplified version of the following one:
for any $\gamma>1$, in the event $\| \bxi^{\top}\bbfX\| \le n\lambda/\gamma$, it holds
\begin{equation}\label{SOI2Matrix}
\ell_{n}(\hatbfB^{\rm EWA},\bfBz) \le \inf_{\substack{\bfB\in\bbfM \\ \calP \in \mathscr{P}}} \bigg\{\ell_{n}(\bfB,\bfBz)
+4\lambda\| \calP_{\barbfB,J^c}(\barbfB) \|_{1} +
\frac{\lambda^{2}(\gamma+1)^{2} |J|}{\gamma^{2}\kappa_{\barbfB,J,(\gamma+1)/(\gamma-1)}}\bigg\}
+2 H(\tau),
\end{equation}
where $H$ is defined by \eqref{matrixH}. This inequality as well as \Cref{OI1Matrix} is proved in \Cref{sec:proofs}.

\subsection{Pseudo-posterior concentration}


In what follows, we state the result on the pseudo-posterior concentration in the matrix case.
Akin to the vector case, one of the main building blocks is \cite[Theorem~1.1]{Bobkov11}, see
\Cref{propBobkov} above.  Since the potential $V_n$ in \eqref{potentialMatrix} is convex,  the proposition
applies and implies that, for every $t>0$,
\begin{equation}
\label{eqBobkovMatrix}
\hat\pi_n\Big(\bfB :\ V_n(\bfB) \leq  \int_{\bbfM}V_n(\bfU)\,\hat\pi_n(\bfU)\, {\rm d}\bfU + \tau \sqrt{m_1 m_2} t\Big) \ge 1-2e^{-t/16}.
\end{equation}
After some nontrivial algebra, this allows us to show that a risk bound similar to \eqref{OI1Matrix}
holds not only for the pseudo-posterior-mean $\bfBEWA$, but also for any matrix $\bfB$ randomly
sampled from $\hat\pi_n$.

\begin{theorem}
\label{theoPosteriorConcentrationMatrix}
Let data be generated by model \eqref{modelMatrix} with  $\bxi\sim\mathcal N(\mathbf{0}_n,\sigma^2\bfI_n)$ and
let the quality of an estimator be measured by the squared prediction loss \eqref{mloss}. Assume that the tuning
parameter $\lambda$ satisfies $\lambda\ge 2\sigma v_\calX \{{\nicefrac{2}{n}\log((m_1 + m_2)/\delta)}\}^{1/2}$, for some
$\delta\in(0,1)$.  Then, with probability at least $1-\delta$, the pseudo-posterior $\hat\pi_n$ with the nuclear-norm prior
defined by {\rm (\ref{potentialMatrix})} is such that the probability
\begin{equation}\label{eqThmConcentrationBound}
\hat\pi_n\bigg(\bfB:
\ell_{n}(\bfB,\bfBz) \le \inf_{\substack{\barbfB\in\bbfM \\ J\subset [\rk(\barbfB)]}}
\bigg\{\ell_{n}(\barbfB,\bfBz)+4\lambda\| \calP_{\barbfB,J^c}(\barbfB) \|_{1}+
\frac{ 9 \lambda^{2} |J|}{4\kappa_{\barbfB,J,3}}\bigg\} + 8 m_1 m_2 \tau \bigg)
\end{equation}
is larger than $1- 2 e^{-\sqrt{m_1m_2}/16}$.
\end{theorem}

We postpone the proof of Theorem \ref{theoPosteriorConcentrationMatrix} to Section \ref{sec:proofs}. One can deduce from
\Cref{theoPosteriorConcentrationMatrix} that if the temperature parameter $\tau$ is sufficiently small, for instance,
$\tau\le \lambda^2/(m_1m_2)$, then a random matrix sampled from the pseudo-posterior $\hat\pi_n$ satisfies nearly the
same oracle inequality as the nuclear-norm penalized least-squares estimator. Indeed, the term $8m_1m_2\tau$, which is the
only difference between the two upper bounds, is in this case negligible with respect to the term involving $\lambda^2$.

\section{Conclusions}\label{sec:concl}

We have considered the model of regression with fixed design and established risk bounds for the exponentially weighted
aggregate with the Laplace prior. This class of estimators encompasses important particular cases such as the lasso
and the Bayesian lasso. The risk bounds established in the present work exhibit a range of values for the temperature
parameter for which the EWA with the Laplace prior has a risk bound of the same order as the lasso. This offers a valuable
complement to the negative results by \cite{Castillo1}, which show that the Bayesian lasso is not rate-optimal in the
sparsity scenario. Note that the Bayesian lasso corresponds to the EWA with the Laplace prior for the temperature parameter
$\tau = \sigma^2/n$, where $\sigma^2$ is the variance of the noise. Our results imply that in order to get rate-optimality
in the sparsity scenario, it is sufficient to choose $\tau$ smaller than $\sigma^2/(np)$.

We have extended the result outlined in the previous paragraph in two directions. First, we have shown that
one can replace the pseudo-posterior mean by any random sample from the pseudo-posterior distribution. This eventually
increases the risk by a negligible additional term, but might be useful from a computational point of view. Second, we
have established risk bounds of the same flavour in the case of trace-regression, when the unknown parameter is a nearly
low-rank large matrix. This result extends those of \citep{KLT11} and unifies the risks bounds leading to the ``slow''
and ``fast'' rates. Furthermore, our result offers an interpolation between these two extreme cases, see the discussion
following \Cref{OI1Matrix}.

With some additional work, all the results established in the present work can be extended to the model of regression
with random design. Furthermore, the case of a partially labelled sample can be handled by coupling the methodology of
the present work with that of \citep{BDGP}. An interesting line of future research is to apply our approach to other priors
constructed from convex penalties such as the mixed $\ell_1/\ell_2$-norm used in the group-lasso \citep{Yuan06}, or
the weighted $\ell_1$-norm of ordered entries used in the slope \citep{SLOPE1}. Another highly relevant and challenging
topic for future work will be to investigate the computational complexity of various methods for approximating the
pseudo-posterior mean or for drawing a sample from the pseudo-posterior density.

\section{Proofs}\label{sec:proofs}

\subsection{Proof of the oracle inequality of Theorem \ref{OI1}}

To ease notation, throughout this section we write $\hatbbeta$ instead of $\bEWA$. Furthermore, for a function
$h:\R^p\to\R$, we often write $\int h\,\hat\pi_n$ instead of $\int_{\R^p} h(\bu)\,\hat\pi_n(\bu)\,{\rm d}\bu$. We
split the proof into three steps. The first step, carried out in \Cref{propfond}, consists in deriving an initial
upper bound on the prediction loss from the fundamental inequality stated in \eqref{lemConcentration}.
The second step, performed in \Cref{lem:2}, shares many common features with the analogous developments for the lasso
and provides a proof of \eqref{SOI2}. Finally, the third step is a standard  bound of the probability of the
event $\mce_\gamma=\{\| \mathbf X^{\top}\bxi\|_{\infty}\le n\lambda/\gamma\}$ based on the union bound and  properties
of the Gaussian distribution.

\begin{lemma}\label{propfond}
For any $\barbbeta\in\R^{p}$,we have
$$
\ell_n(\hatbbeta,\bbetaz)
    \le \ell_n(\barbbeta,\bbetaz) + \frac2n{ \| \bfX^{\top}\bxi \|_{\infty} \|  \hatbbeta - \barbbeta \|_{1}}+
        2\lambda ( \| \barbbeta \|_{1} - \| \hatbbeta \|_{1} ) + 2H(\tau) -\frac1n\|\bfX(\barbbeta-\hatbbeta)\|_2^2.
$$
\end{lemma}

\begin{proof}
On the one hand, inequality \eqref{lemConcentration} can be rewritten as
\begin{equation}\label{eq:5}
V_n(\hatbbeta)
    \le V_n(\barbbeta) + \underbrace{V_n(\hatbbeta)-\int_{\R^p} V_n(\bu)\, \hat\pi_n(\bu)\, {\rm d}\bu + p \tau -
    \frac{1}{2n} \int_{\mathbb{R}^{p}} \| \bfX (\bu - \barbbeta) \|_{2}^{2}\,\hat\pi_n(\bu)\, {\rm d}\bu}_{:=A}.
\end{equation}
On the other hand, one can check that
\begin{align}
V_n(\hatbbeta)-\int_{\R^p} V_n(\bu)\,\hat\pi_n(\bu)\,{\rm d}\bu
    &= \frac1{2n}\|\bfX\hatbbeta \|_2^2+\lambda\|\hatbbeta\|_1-\int_{\R^p}\Big(\frac1{2n}\|\bfX\bu \|_2^2+
        \lambda\|\bu\|_1\Big)\,\hat\pi_n(\bu)\,{\rm d}\bu,\\
\int_{\mathbb{R}^{p}} \| \bfX (\bu - \barbbeta) \big\|_{2}^{2}\,\hat\pi_n(\bu)\,{\rm d}\bu
    &= \|\bfX(\barbbeta-\hatbbeta)\|_2^2+\int_{\mathbb{R}^{p}} \| \bfX\bu \|_{2}^{2}\,\hat\pi_n(\bu)\, {\rm d}\bu
        - \| \bfX\hatbbeta \|_{2}^{2}.
\end{align}
These inequalities, combined with the definition of $H$, given in \eqref{H}, yield
\begin{align}
A &= \frac1{n}\|\bfX\hatbbeta \|_2^2+\lambda\|\hatbbeta\|_1-\int_{\R^p}\Big(\frac1{n}\|\bfX\bu \|_2^2 +
        \lambda\|\bu\|_1\Big)\,\hat\pi_n(\bu)\,{\rm d}\bu + p\tau
        -\frac1{2n}\|\bfX(\barbbeta-\hatbbeta)\|_2^2\\
  &= H(\tau)- \frac1{2n}\|\bfX(\barbbeta-\hatbbeta)\|_2^2.\label{eq:6}
\end{align}
Finally, using the definitions of the prediction loss $\ell_n$ and the potential $V_n$, we get
\begin{equation}\label{eq:7}
\ell_n(\hatbbeta,\betaz) - \ell_n(\barbbeta,\betaz)  =  2\big(V_n(\hatbbeta) - V_n(\barbbeta) \big) +
\frac2n\,\bxi^\top\bfX(\hatbbeta-\barbbeta) + 2\lambda( \| \barbbeta \|_{1} - \| \hatbbeta \|_{1} ).
\end{equation}
In view of the duality inequality, the term $\bxi^\top\bfX(\hatbbeta-\barbbeta)$ is upper bounded in absolute value by
$\|\bfX^\top\bxi\|_\infty\|\hatbbeta-\barbbeta\|_1$. Inserting this inequality and \eqref{eq:5} in \eqref{eq:7} and using relation
\eqref{eq:6}, we get the claim of the lemma.
\end{proof}

According to \Cref{propfond}, in the event $\mce_\gamma=\{\| \mathbf X^{\top}\bxi\|_{\infty}\le n\lambda/\gamma\}$, we have
\begin{equation}
\label{OI1e1}
\ell_n(\hatbbeta,\bbetaz)
    \le \ell_n(\barbbeta,\bbetaz) + \frac{2\lambda}{\gamma}(\| \hatbbeta-\barbbeta\|_{1}+
    \gamma\|\barbbeta\|_{1}-\gamma\|\hatbbeta\|_{1}) + 2H(\tau) -\frac1n\|\bfX(\barbbeta-\hatbbeta)\|_2^2.
\end{equation}

\begin{lemma}\label{lem:2}
For every $J\subset [p]$, we have
$$
\frac{2\lambda}{\gamma}(\| \hatbbeta-\barbbeta\|_{1}+
    \gamma\|\barbbeta\|_{1}-\gamma\|\hatbbeta\|_{1}) -\frac1n\|\bfX(\barbbeta-\hatbbeta)\|_2^2
    \le 4\lambda\|\barbbeta_{J^{c}}\|_{1} +
\frac{\lambda^{2}(\gamma+1)^{2}| J|}{\gamma^{2}\kappa_{J,(\gamma+1)/(\gamma-1)}}.
$$
\end{lemma}

This lemma is essentially a copy of Proposition 2 in \citep{BDGP}. We provide here its proof for the sake of self-containedness.
\begin{proof}
Let us fix a $J\subset\{1,\dots,p\}$ and set $\bu = \hatbbeta-\barbbeta$. We have
\begin{equation}
\| \hatbbeta-\barbbeta\|_{1}+\gamma\|\barbbeta\|_{1}-\gamma\|\hatbbeta\|_{1}
=\| \bu_{J}\|_{1}+ \| \bu_{J^{c}}\|_{1}+\gamma\|\barbbeta_{J}\|_{1}
+ \gamma\|\barbbeta_{J^{c}}\|_{1}-\gamma\|\hatbbeta_{J}\|_{1}-\gamma\|\hatbbeta_{J^{c}}\|_{1}.
\label{OI1e2}
\end{equation}
Using inequalities $\|\barbbeta_{J}\|_{1} - \|\hatbbeta_{J}\|_{1} \le \|\bu_{J}\|_{1}$ and
$\|\hatbbeta_{J^{c}}\|_{1} \ge \|\bu_{J^{c}}\|_{1} - \|\barbbeta_{J^{c}}\|_{1}$, we deduce
from equation \eqref{OI1e2} that
\begin{eqnarray}
\| \hatbbeta - \barbbeta \|_{1} + \gamma \| \barbbeta \|_{1} - \gamma \|\hatbbeta \|_{1}
    \le  (\gamma+1) \| \bu_{J} \|_{1} - (\gamma-1) \| \bu_{J^{c}} \|_{1}
        +2\gamma \| \barbbeta_{J^{c}} \|_{1}.
\label{OI1e3}
\end{eqnarray}
Now, by definition of the compatibility factor $\kappa_{J,c}$ given by equation \eqref{CF}, we obtain
\begin{equation}
\label{OI1e6}
\| \bu_{J} \|_{1} - \frac{\gamma-1}{\gamma+1} \| \bu_{J^{c}} \|_{1}
\le \bigg(\frac{| J|\| \bfX \bu \|^{2}_{2}}{n\kappa_{J,(\gamma+1)/(\gamma-1)}}\bigg)^{1/2}.
\end{equation}
Hence, inequalities \eqref{OI1e3} end \eqref{OI1e6} imply that
\begin{align}
\frac{2\lambda}{\gamma}(\| \hatbbeta-\barbbeta\|_{1}+\gamma\|\barbbeta\|_{1}-
\gamma\|\hatbbeta\|_{1}) -\frac1n\|\bfX(\barbbeta-\hatbbeta)\|_2^2
    \le 4\lambda \| \barbbeta_{J^{c}} \|_{1} + 2ab -a^{2},
\label{OI1e7}
\end{align}
where we have used the notation $a^2 = \| \bfX \bu\|^2_{2}/n$ and
$b^2 = \frac{\lambda^2(\gamma+1)^2|J|}{\gamma^2\kappa_{J,(\gamma+1)/(\gamma-1)}}$.
Finally, noticing that
$$
2ab -a^{2}\le b^2 =
\frac{\lambda^{2}(\gamma+1)^{2}| J|}{\gamma^{2}\kappa_{J,(\gamma+1)/(\gamma-1)}},
$$
we get the claim of the lemma.
\end{proof}

Combining the claims of the previous lemmas and taking the minimum with respect to $J$ and $\barbbeta$,
we obtain that the inequality
\begin{equation}\label{eq:8}
\ell_{n} \big( \hatbbeta,\betaz \big) \le \inf_{\substack{\barbbeta \in \R^p \\ J \subset [p]}}
\bigg\{\ell_{n}\big(\barbbeta,\betaz \big)+4\lambda\|\barbbeta_{J^{c}}\|_{1} +
\frac{\lambda^{2}(\gamma+1)^{2}| J|}{\gamma^{2}\kappa_{J,(\gamma+1)/(\gamma-1)}}\bigg\}+2H(\tau)
\end{equation}
holds in the event $\mce_\gamma$. The third and the last step of the proof consists in assessing the
probability of this event.

\begin{lemma}
\label{lemcor1}
If $\bfX =(\bx^{1},\ldots,\bx^{p})$ is a $n\times p$ deterministic matrix with columns $\bx^{j}$ satisfying
$\|\bx^{j}\|_2^2\le n$ and if $\bxi\sim \mathcal N(\mathbf 0_n,\sigma^{2} \bfI_n)$, then, for all $\e>0$,
\begin{equation}
    \prob\big(\| \mathbf X^{\top}\bxi\|_{\infty}>n\e\big)\le p\,\exp\big(-{n\e^{2}}/{(2\sigma^{2})}\big).
\end{equation}
\end{lemma}

\begin{proof} By the union bound, we get
\begin{equation}
\prob\big(\| \bfX^{\top}\bxi\|_{\infty}>n\e\big)=\prob\left(\max_{j\in[p]}| \bxi^{\top} \bx^{j}|>n\e\right)
\le \sum_{i=1}^{p}\prob\big(| \bxi^{\top}\bx^{j}|>n\e\big).
\end{equation}
Then, noticing that for each $j\in[p]$ the random variable $\bxi^{\top}\bx^{j}$ is distributed according to
$\mathcal N(0,\sigma^{2}\| \bx^{j}\|_2^{2})$, we deduce that
$$
\prob\left(\| \bfX^\top\bxi\|_{\infty}>n\e\right)\le 2\sum_{j=1}^{p}\int^{+\infty}_{n\e/(\sigma\| \bx^{j}\|_{2})}\phi(u)\,{\rm d}u,
$$
where $\phi$ stands for the probability density function of the standard Gaussian distribution. Finally,
by using the inequality $\int_{x}^{+\infty}\phi(u)\,{\rm d}u\le\nicefrac{1}{2}\exp(-{x^{2}}/{2})$ that
holds for every $x>0$, we obtain the result.
\end{proof}
A proof of \Cref{OI1} can be deduced from the three previous lemmas as follows. Choosing $\gamma=2$ and
$\e ={\lambda/2} \ge\sigma\sqrt{(2/n)\log(p/\delta)}$ in \Cref{lemcor1}, we get that the event $\mce_\gamma$
has a probability at least $1-\delta$. Furthermore, on this event, we have already established
inequality \eqref{eq:8}. Finally, upper bounding $H(\tau)$ by $p\tau$ leads to the claim of the theorem.

\subsection{Proof of the concentration property of Theorem \ref{theoPosteriorConcentration}}

Let us introduce the set $\mcb = \{\bbeta\in\R^p: V_n(\bbeta) \le \int V_n\,\hat\pi_n + p\tau\}$. Applying \Cref{propBobkov} with $t=\sqrt{p}$,
we get $\hat\pi_n(\mcb)\ge 1-2e^{-\sqrt{p}/16}$. To prove \Cref{theoPosteriorConcentration}, it is sufficient to check that in the event $\mce_\gamma$
(in particular, with $\gamma=2$), every vector $\bbeta$ from $\mcb$ satisfies the inequality
\begin{equation}\label{eq:9}
\ell_{n}(\bbeta,\betaz ) \le \inf_{\substack{\bbeta\in\R^{p} \\ J\subset[p]}}
\bigg\{\ell_{n}(\barbbeta,\betaz )+4\lambda \|\barbbeta_{J^{c}}\|_{1}+
\frac{9\lambda^{2}| J|}{2\kappa_{J,3}}\bigg\} + 8p\tau.
\end{equation}
In the rest of this proof, $\bbeta$ is always a vector from $\mcb$. In view of \eqref{lemConcentration}, it satisfies
\begin{equation}\label{eq:10}
V_n(\bbeta) \le 2p\tau + V_n(\barbbeta)-
\frac{1}{2n} \int_{\mathbb{R}^{p}} \| \bfX (\bu - \barbbeta) \|_{2}^{2}\,\hat\pi_n(\bu)\, {\rm d}\bu.
\end{equation}
Note that \eqref{eq:10} holds for every $\barbbeta\in\R^p$. Therefore, it also holds for $\barbbeta=\bbeta$ and yields
\begin{equation}\label{eq:13}
\frac{1}{n} \int_{\mathbb{R}^{p}} \| \bfX (\bu - \bbeta) \|_{2}^{2}\,\hat\pi_n(\bu)\, {\rm d}\bu \le  4p\tau.
\end{equation}
In addition, we have
\begin{equation}\label{eq:11}
\ell_n(\bbeta,\betaz) - \ell_n(\barbbeta,\betaz)  =  2\big(V_n(\bbeta) - V_n(\barbbeta) \big) +
    \frac2n\,\bxi^\top\bfX(\bbeta-\barbbeta) + 2\lambda( \| \barbbeta \|_{1} - \| \bbeta \|_{1} ).
\end{equation}
Combining \eqref{eq:10}, \eqref{eq:11} and the duality inequality, we get that in $\mce_\gamma$,
\begin{align}
\ell_n(\bbeta,\betaz) - \ell_n(\barbbeta,\betaz)
    & \le  4p\tau - \frac{1}{n} \int_{\mathbb{R}^{p}} \| \bfX (\bu -
    \barbbeta) \|_{2}^{2}\,\hat\pi_n(\bu)\, {\rm d}\bu\\
    &\qquad + \frac{2\lambda}{\gamma}\,\|\bbeta-\barbbeta\|_1 + 2\lambda( \| \barbbeta \|_{1} -
    \| \bbeta \|_{1} ).\label{eq:12}
\end{align}
We use now the inequality $\|\bfX(\bu-\barbbeta)\|_2^2\ge \nicefrac12\|\bfX(\bbeta-\barbbeta)\|_2^2-\|\bfX(\bu-\bbeta)\|_2^2$,
in conjunction with \eqref{eq:13}, to deduce from \eqref{eq:12} that
\begin{align}
\ell_n(\bbeta,\betaz) - \ell_n(\barbbeta,\betaz)
    & \le  8p\tau + \frac{2\lambda}{\gamma}\,\|\bbeta-\barbbeta\|_1 + 2\lambda( \| \barbbeta \|_{1} -
    \| \bbeta \|_{1} ) - \frac{1}{2n} \| \bfX (\bbeta -\barbbeta) \|_{2}^{2}.
\end{align}
We can apply now \Cref{lem:2} with $\bbeta$ instead of $\hatbbeta$ and $\bfX/\sqrt{2}$ instead of $\bfX$ in order to get
the claim of \Cref{theoPosteriorConcentration}.

\subsection{Proof of Proposition \ref{prop:evalH}}

For the sake of simplicity, we abbreviate $\hat\bbeta=\bEWA$ and $\hat\bbeta\/^{0}=\bbetaLS$ throughout the proof. In particular, notation $\hat\beta_j$ (resp. $\hat\beta{}^{0}_j$) will refer to the $j$-th entry of $\bEWA$ (resp. $\bbetaLS$). First, observe that one can write the posterior density as $\hat\pi(\bu)\propto \exp(-\bar V_n(\bu)/\tau)$ with
\begin{align}
\bar V_n({\bu})&  = V_n(\bu) -\frac{1}{2n}\|\by\|^2+\frac{1}{2}\|\hat\bfSigma_n^{1/2}\hat\bbeta{}^{0}\|^2_2\\
& = \frac1{2}\|\hat\bfSigma_n^{1/2}(\bu-\hat\bbeta{}^{0})\|_2^2+\lambda\|\bu\|_1.
\label{vnbar}
\end{align}
On the one hand, the integration by parts formula yields
$$
\int_{\RR^p} [\bu^\top \nabla \bar V_n(\bu)]\,\hat\pi(\bu)\,{\rm d}\bu = -\tau\int_{\RR^p} \bu^\top \nabla \hat\pi(\bu)\,{\rm d}\bu = p\tau.
$$
On the other hand,  the expression of $\bar V_n(\bu)$ written in \eqref{vnbar} leads directly to
\begin{align}
\int_{\RR^p} [\bu^\top \nabla \bar V_n(\bu)]\,\hat\pi(\bu)\,{\rm d}\bu &=  \int_{\RR^p} G(\bu)\,\hat\pi(\bu)\,{\rm d}\bu
-\hat\bbeta^\top\hat\bfSigma_n\hat\bbeta{}^{0},
\end{align}
where we recall that $G(\bu)=\|\bfX\bu\|^{2}_{2}/n+\lambda\|\bu\|_{1}=\|\hat\bfSigma_n^{1/2}\bu\|^2_2+\lambda\|\bu\|_{1}$. This yields
$$
\int_{\RR^p} G(\bu)\,\hat\pi(\bu)\,{\rm d}\bu = p\tau+\hat\bbeta^\top\hat\bfSigma_n\hat\bbeta{}^{0},
$$
and, hence,
\begin{align}
H(\tau)
    &= p\tau - \frac1n\int_{\RR^p} G(\bu)\hat\pi(\bu)\,{\rm d}\bu +\|\hat\bfSigma_n^{1/2}\hat\bbeta\|_2^2+\lambda\|\hat\bbeta\|_1\\
    &= \|\hat\bfSigma_n^{1/2}\hat\bbeta\|_2^2+\lambda\|\hat\bbeta\|_1-\hat\bbeta^\top\hat\bfSigma_n\hat\bbeta{}^{0},
    \label{exprH}
\end{align}
which proves the first claim of Proposition \ref{prop:evalH}. Let us now consider the case where $\hat\bfSigma_n = \bfI_p$. Then, recalling
the definition of $\bar V_n(\bu)$ in \eqref{vnbar}, a straightforward calculation reveals that
\begin{align}
\bar V_n(\bu)&= \frac{\lambda^2p}{2}+\sum_{j=1}^p\left[\frac{1}{2}\left(u_j-\hat\beta^{0}_j+\lambda\sgn(u_j)\right)^2+\lambda\hat\beta^{0}_j\sgn(u_j)\right].
\end{align}
Hence, we deduce that $\hat\pi(\bu)=\prod_{j=1}^{p}\hat\pi_j(u_j)$ where
\begin{equation}
\hat\pi_j(t)\propto \exp\Big(-\frac{1}{2\tau}(t-\hat\beta^{0}_j+\lambda\sgn(t))^2-\frac{\lambda}{\tau}\hat\beta^{0}_j\sgn(t)\Big).
\end{equation}
Next, let $\varphi(t)=\int_{t}^{+\infty}\phi(x){\rm d}x$ where $\phi$ denotes the density function of the standard normal distribution. For a fixed $j\in[p]$,
we consider the abbreviations $a=\lambda/\sqrt{\tau}$ and $b=\hat\beta^0_j/\sqrt{\tau}$. Then, the change of variable $u=t/\sqrt{\tau}$ in the first integral below,
together with the observation that $\sgn(t)=\sgn(t/\sqrt{\tau})$ for all real $t$, leads to
\begin{align}
\hat\beta_j  = \int t\,\hat\pi_{j}(t)\,{\rm d}t
        & = \sqrt{\tau}\,\frac{\int u \exp\{-\frac12(u-b+a\sgn(u))^2-ab\sgn(u)\}\,{\rm d}u}
            {\int  \exp\{-\frac12(u-b+a\sgn(u))^2-ab\sgn(u)\}\,{\rm d}u}\\
        & = \sqrt{\tau}\,\frac{(a+b)e^{ab}\varphi(a+b)-(a-b)e^{-ab}\varphi(a-b)}{e^{ab}\varphi(a+b)+e^{-ab}\varphi(a-b)}\\
        & = \sqrt{\tau}\,\sgn(b)\frac{(a+\vert b\vert)e^{a\vert b\vert}\varphi(a+\vert b \vert)-(a-\vert b\vert)e^{-a\vert b\vert}
            \varphi(a-\vert b\vert)}{e^{a\vert b\vert}\varphi(a+\vert b \vert)+e^{-a\vert b\vert}\varphi(a-\vert b\vert)}\\
        & = \hat\beta^{0}_{j}+\lambda\sgn(\hat\beta^{0}_{j})\frac{e^{a\vert b\vert}\varphi(a+\vert b \vert)-
            e^{-a\vert b\vert}\varphi(a-\vert b\vert)}{e^{a\vert b\vert}\varphi(a+\vert b \vert)+e^{-a\vert b\vert}\varphi(a-\vert b\vert)}\\
        & = \hat\beta^{0}_{j}+\lambda\sgn(\hat\beta^{0}_{j})\frac{\Psi(a+\vert b \vert)-\Psi(a-\vert b\vert)}{\Psi(a+\vert b \vert)+\Psi(a-\vert b\vert)},
\label{prop1:e1}
\end{align}
where $\Psi(t)=e^{t^{2}/2}\varphi(t)$. In other terms, noticing that $\Psi_{\tau}(t)=\Psi(t/\sqrt{\tau})$, we have obtained
\begin{equation}
\hat\beta_j = \sgn(\hat\beta{}^{0}_{j})\left(|\hat\beta{}^{0}_{j}|-\lambda w(\tau,\lambda,|\hat\beta{}^{0}_{j}|)\right),
\label{proofbewaj}
\end{equation}
where we have denoted $w(\tau,\lambda,t)=(\Psi_{\tau}(\lambda-t)-\Psi_{\tau}(\lambda+t))/(\Psi_{\tau}(\lambda-t)+\Psi_{\tau}(\lambda+t))$.
Finally, injecting \eqref{proofbewaj} in \eqref{exprH} leads easily to the desired expression for $H$.

\subsection{Proofs for Stein's unbiased risk estimate \eqref{riskEstim}}

In what follows, we denote $\hatbbeta=\hatbbeta{}^{\rm EWA}$ for brevity. The dependance on $\by$ will sometimes be
made explicit in the proof for clarity. Below, it is understood that all gradients are taken with respect to
variable $\by$ and that $\partial_i$ refers to the $i$-th element of the gradient. In addition, for every
function $h:\R^n\to\R^n$, we use the notation $\nabla\cdot h$ for the divergence operator
$\sum_i \partial_i h_i$. Finally, function $f$ will refer to the non-normalized pseudo-posterior,
$f(\bbeta,\by) = \exp(-{V_n(\bbeta, \by)}/{\tau})$, and $g$ to its integral (the normalizing constant), {\it i.e.}
\begin{equation}
\label{g}
g(\by) = \int_{\R^p}f(\bbeta,\by)\,{\rm d}\bbeta.
\end{equation}
According to Stein's formula, an unbiased estimate of the risk of $\hatbbeta$---under Gaussian
noise---is given by
\begin{equation}
\label{gsure}
\hat R^{\rm EWA}(\lambda,\tau) =
\frac1n\| \by - \bfX \hatbbeta\|_2^{2} - \frac{\sigma^2}{n}  + \frac{2\sigma^2}{n} \nabla\cdot (\bfX\hatbbeta).
\end{equation}
Therefore, to prove \eqref{riskEstim}, we need  only to show that
\begin{eqnarray}
\label{nablaBeta}
\nabla \hatbbeta(\by) &=& \frac{\Cov_{\hat\pi}(\bbeta)}{n\tau} \bfX^{\top} \in \R^{p\times n}.
\end{eqnarray}
Indeed, this will imply that
$$
\nabla\cdot (\bfX\hatbbeta) = \sum_i \bx_i^\top \partial_i\hatbbeta(\by) = \frac1{n\tau}
\sum_i \bx_i^\top\Cov_{\hat\pi}(\bbeta)\bx_i = \frac1{n\tau}\int_{\R^p} \|\bfX(\bbeta-\hatbbeta)\|_2^2\,\hat\pi_n({\rm d}\bbeta),
$$
which, combined with \eqref{gsure}, leads to \eqref{riskEstim}. To do so, we proceed in two steps. First, we prove that
\begin{eqnarray}
\label{nablaIPi}
\partial_i \hat\pi(\bbeta,\by) &=&\frac{\bx_i}{n\tau} (\bbeta-\hatbbeta(\by)) \hat\pi(\bbeta,\by).
\end{eqnarray}
Secondly, we show that
\begin{eqnarray}
\label{nablaIBeta}
\partial_i \hat\beta(\by) &=&  \frac{\Cov_{\hat\pi}(\bbeta)}{n\tau} \bx_i^{\top}.
\end{eqnarray}
Given the notations introduced above, we have
\begin{equation}
\label{fOnG}
\hat\pi(\bbeta,\by) = \frac{f(\bbeta, \by)}{g(\by)}.
\end{equation}
Then, notice that
\begin{equation}
\label{fPrime}
{\partial_i f}(\bbeta,\by) = -\bigg(\frac{y_i - \bx_i^\top\bbeta }{n\tau}\bigg) f(\bbeta, \by).
\end{equation}
Hence, combining \eqref{fOnG} and \eqref{fPrime} yields,
\begin{eqnarray}
\label{fPrimeOnG}
\frac{{\partial_i f}(\bbeta,\by)}{g(\by)} &=& - \frac{1}{n\tau}(y_i  - \bx_i^\top \bbeta)\,\hat\pi(\bbeta,\by).
\end{eqnarray}
Moreover, using one more time \eqref{fPrime}, we get
\begin{eqnarray}
\frac{{\partial_i g}(\by)}{g(\by)} &=& \frac{\int {\partial_i f}(\bbeta,\by){\rm d}\bbeta}{g(\by)}
\nonumber\\
&=& - \frac{1}{n\tau}\frac{\int (y_i - \bx_i^\top\bbeta) f(\bbeta, \by){\rm d}\bbeta}{g(\by)}
\nonumber\\
&=& \frac{\bx_i^\top}{n\tau} \frac{\int \bbeta f(\bbeta, \by){\rm d}\bbeta}{g(\by)} - \frac{y_i}{n\tau}
\nonumber\\
&=& {\frac{1}{n\tau} (\bx_i^\top\hatbbeta(\by) - y_i),}
\label{gPrimeOng}
\end{eqnarray}
where \eqref{gPrimeOng} follows from the definition of $\hatbbeta$, $f$ and $g$. With these remarks in mind, observe that
\begin{align}
\partial_i \hat\pi(\bbeta, \by)
 &= \frac{{\partial_i f}(\bbeta,\by)  g(\by) - f(\bbeta,\by)  {\partial_i g}(\by)}{g(\by)^{2}}\\
 &= \frac{{\partial_i f}(\bbeta,\by)}{ g(\by)} - {\hat\pi(\bbeta,\by)}\frac{{\partial_i g}(\by)}{g(\by)}\\
 &= \frac{\bx_i^\top}{n\tau}(\bbeta  -  \hatbbeta(\by))\,\hat\pi(\bbeta,\by),
\end{align}
where the last line follows easily by combining \eqref{fPrimeOnG} and \eqref{gPrimeOng}. We have therefore proved \eqref{nablaIPi}
and now proceed to showing \eqref{nablaIBeta}. To that aim, we write
\begin{equation}
\partial_i \hatbbeta(\by) =\partial_i \int_{\R^p}\bbeta\,\hat{\pi}(\bbeta, \by)\,{\rm d}\bbeta
=   \int_{\R^p}\bbeta\, \partial_i \hat{\pi}(\bbeta, \by)\,{\rm d}\bbeta.
\end{equation}
Using  \eqref{nablaIPi} and then transposing the product $\bx_i^\top (\bbeta - \hatbbeta(\by)) \in  \R$ we have,
\begin{eqnarray}
\partial_i \hatbbeta(\by) & = & \frac{1}{n\tau} \int_{\R^p}\bbeta (  \bx_i^\top(\bbeta-\hatbbeta(\by)) \hat\pi(\bbeta,\by)) {\rm d}\bbeta
\nonumber\\
& = & \frac{1}{n\tau}\int_{\R^p} ( \bbeta \bbeta^{\top}  - \bbeta  \hatbbeta(\by)^{\top}  )\bx_i\,\hat\pi(\bbeta,\by) {\rm d}\bbeta
\nonumber\\
& = & \frac{1}{n\tau}\bigg(\int_{\R^p}  \bbeta \bbeta ^{\top}\hat\pi(\bbeta,\by) {\rm d}{\bbeta} -  \hatbbeta(\by) \hatbbeta(\by) ^{\top}
\bigg)\, \bx_i,
\nonumber
\end{eqnarray}
which is equivalent to \eqref{nablaIBeta} and concludes the proof of \eqref{riskEstim}.

\subsection{Proof of the results in the matrix case}
To ease notation, throughout this section we write $\hatbfB$ instead of $\bfBEWA$. Furthermore, for a function
$h:\bbfM\to\R$, we often use the notation $\int h\,\hat\pi_n$ or $\int_{\mcM} h(\bfU)\,\hat\pi_n({\rm d}\bfU)$ instead of
$\int_{\bbfM} h(\bfU)\,\hat\pi_n(\bfU)\,{\rm d}\bfU$.
In this section, we prove \Cref{OI1Matrix} and \Cref{theoPosteriorConcentrationMatrix}. To do so, we state and
prove \Cref{propFundamentalCalculus} as well as \Cref{propVariancePosteriorBoundMatrix} that will be used
throughout the proofs.

\Cref{propFundamentalCalculus} is an extension of the fundamental theorem of calculus in the case of
locally-Lipschitz functions. It will be very useful to work with any (pseudo-)posterior of the form
$\hat\pi_n$ corresponding to convex penalties.

Let us first recall that a function $f:\R\to\R$ is called locally-Lipschitz-continuous, or locally-Lipschitz,
if it is Lipschitz-continuous on any bounded interval. Clearly, any locally-Lipschitz function is absolutely
continuous (in the sense of Definition 7.17 in \citet{Rud87}) and, therefore, is almost everywhere (with respect
to the Lebesgue measure) differentiable.

\begin{prop}
\label{propFundamentalCalculus}
For any locally-Lipschitz function $f$ such that $\lim_{|x|\to\infty} f(x) = 0$ and $f'\in L^1{(\R)}$, we have
$$
\int_{\R} f'(x) {\rm d} x = 0.
$$
\end{prop}

\begin{proof} The result of \cite[Theorem 7.20]{Rud87} implies that for any $a>0$,
$$
\int_{-a}^{a} f'(x)\, {\rm d} x = f(a) - f(-a).
$$
Since, by assumption, the derivative $f'$  is absolutely integrable over $\R$, we have
$$
\int_{\R} f'(x)\, {\rm d} x = \lim_{a \to +\infty}
\int_{-a}^{a} f'(x)\, {\rm d} x = \lim_{a\to +\infty} \big(f(a) - f(-a)\big)=0.
$$
This completes the proof.
\end{proof}

\begin{cor}
\label{corFundamentalCalculus1}
 Let $\mathbf{0}_{m_1, m_2}$ be the null element of $\bbfM$. Then
 $$\int_{\mcM}\nabla \hat\pi_n(\bfU)\, {\rm d}\bfU = \mathbf{0}_{m_1, m_2}.$$
\end{cor}

\begin{proof}
We want to prove that $\int_{\mcM}[\partial_{\bfU_{\bd}} \hat\pi_n(\bfU)]\, {\rm d}\bfU = 0$ for any
$d := (k,l) \in[m_1]\times [m_2]$. To this end, we will simply prove that
$\int_{\R}[\partial_{\bfU_{\bd}} \hat\pi_n(\bfU)]\, {\rm d}\bfU_{\bd} = 0$, where the integration is done with respect
to the $\bd$-th entry of $\bfU$ when all the other entries are fixed.

The function $\bfU\mapsto V_n(\bfU)$ is locally-Lipschitz as the sum of a continuously differentiable function
(the quadratic term) and a Lipschitz term (the nuclear norm). We note in passing that any norm in a finite-dimensional
space is Lipschitz continuous thanks to the triangle inequality and the equivalence of norms. In addition, one easily checks
that $\|\bfU\|_1^2\ge \|\bfU\|^2 = \|\bfU^\top\bfU\|\ge \max_{l} (\bfU^\top\bfU)_{l,l}\ge \bfU_{\bd}^2$. This implies that
if $\bfU_{\bd}$ tends to infinity while all the other entries of $\bfU$ remain fixed, the nuclear norm $\|\bfU\|_1$ tends to
infinity\footnote{This assertion can be also established for any other norm using the equivalence of norms in $\bbfM$.}.

As a consequence, the function $\bfU_{\bd}\mapsto \pi_n(\bfU)\propto \exp\{-V_n(\bfU)/\tau\}$ is locally-Lipschitz and
tends to zero when $|\bfU_{\bd}|\to\infty$. This implies that we can apply \Cref{propFundamentalCalculus} and the claim of the corollary follows.
\end{proof}

\begin{cor}
\label{corFundamentalCalculus2}
 With the notation introduced in \Cref{sec:matrix}, we have
 \begin{equation}\label{eq:14}
 \int_{\mcM} \langle \bfU , \nabla V_n(\bfU) \rangle\, \hat\pi_n(\bfU)\,{\rm d}\bfU = \tau m_1  m_2.
 \end{equation}
\end{cor}

\begin{proof}[Proof of Corrolary \ref{corFundamentalCalculus2}]
We first remark that \eqref{eq:14} can be equivalently written as
\begin{equation}
\sum_{d \in [m_1]\times[m_2]}
\int_{\mcM} \bfU_{\bd}\,[\partial_{\bfU_{\bd}}V_n(\bfU)]\, \hat\pi_n(\bfU)\,{\rm d}\bfU = \tau m_1  m_2.
\end{equation}
To establish this identity, it suffices to prove that each integral of the left-hand side is equal to $\tau$.
We have already checked in the proof of \Cref{corFundamentalCalculus1} that the mapping $\bfU_{\bd}\mapsto \pi_n(\bfU)$
is locally-Lipschitz and tends to zero when $\bfU_{\bd}$ tends to infinity. Furthermore, the latter convergence is exponential
so that $\bfU_{\bd}\pi_n(\bfU)$ tends to zero as well, when $\bfU_{\bd}$ tends to infinity.  In view of \Cref{propFundamentalCalculus},
this yields
\begin{align}
\int_{\mcM} \frac{\partial [\bfU_{\bd}\,\hat\pi_n(\bfU)]}{\partial \bfU_{\bd}}\, {\rm d}\bfU  = 0.
\end{align}
Moreover, we remark that
\begin{align}
\frac{\partial [\bfU_{\bd}\, \hat\pi_n(\bfU)]}{\partial \bfU_{\bd}}
        &=  \bfU_{\bd} \frac{\partial\hat\pi_n(\bfU)}{\partial \bfU_{\bd}} + \hat\pi_n(\bfU)
        =  - \bfU_{\bd}\, \frac{\partial V_n(\bfU)}{\tau\partial \bfU_{\bd}}\,\hat\pi_n(\bfU) + \hat\pi_n(\bfU).
\end{align}

Therefore, multiplying by $\tau$ and integrating over $\bbfM$,  we get
$$
\int_{\mcM} \bfU_{\bd}\, \frac{\partial V_n(\bfU)}{\partial \bfU_{\bd}}\,\hat\pi_n(\bfU)\,{\rm d}\bfU = \tau \int_{\mcM} \hat\pi_n(\bfU) {\rm d}\bfU = \tau.
$$
This completes the proof.
\end{proof}

The next Proposition is the matrix analogue of \Cref{variancePosteriorBound}.

\begin{prop}
\label{propVariancePosteriorBoundMatrix}
Let $\hat\pi_n(\bfU)\propto \exp{(-V_n(\bfU)/\tau)}$ be the pseudo-posterior defined
by \eqref{potentialMatrix}. Then, for every $\barbfB\in\bbfM$,
we have
\begin{align}
\int_{\mcM} V_n(\bfU)\,\hat\pi_n({\rm d}\bfU)
        & \le  V_n (\barbfB ) - \frac12\int_{\mcM} \| \barbfB - \bfU \|_{L_2(\bbfX)}^{2}\hat\pi_n({\rm d}\bfU)+m_1m_2\tau.
            \label{lemConcentrationMatrix}
\end{align}
Furthermore,
\begin{equation}
\label{varianceInequalityMatrix}
\int_{\mcM}\| \bfU - \bfBEWA  \|_{L_2(\bbfX)}^{2}\,\hat\pi_n({\rm d}\bfU) \leq  m_1 m_2\tau.
\end{equation}
\end{prop}

\begin{proof}
The convexity of $\bfU \mapsto \| \bfU \|_1$ and the strong convexity of the function $\btheta \mapsto \| \by - \btheta \|_{2}^{2}$ applied in $\btheta = \sum_{i \in [n]} \langle \bfX_i , \bfU \rangle $ imply that for any $\bfU, \barbfB \in \bbfM$,
\begin{equation}\label{eq:15}
V_n\big(\barbfB \big) \geq V_n(\bfU) + \langle \barbfB - \bfU , \nabla V_n(\bfU) \rangle + \frac{1}{2} \|\barbfB - \bfU \|_{L_2(\bbfX)}^{2}.
\end{equation}
In order to prove  \Cref{propVariancePosteriorBoundMatrix} we rely on Corrolaries \ref{corFundamentalCalculus1} and \ref{corFundamentalCalculus2} from \Cref{propFundamentalCalculus}:
\begin{equation}
\label{lemConcentrationRk1Matrix}
\int_{\mcM} \nabla V_n(\bfU) \,\hat\pi_n({\rm d}\bfU)  = 0\qquad\text{and}\qquad
\int_{\mcM} \langle \bfU,  \nabla V_n(\bfU)\rangle \,\hat\pi_n({\rm d}\bfU)  = m_1m_2 \tau.
\end{equation}
We integrate inequality \eqref{eq:15} over $\bbfM$ with respect to the density $\hat\pi_n$ and use equalities \eqref{lemConcentrationRk1Matrix}. This yields
\begin{equation}
V_n\big(\barbfB \big) \geq \int_{\mcM} V_n(\bfU) \,\hat\pi_n({\rm d}\bfU)  - m_1m_2\tau + \frac{1}{2} \int_{\mcM} \|  \barbfB - \bfU  \|_{L_2(\bbfX)}^{2}\,\hat\pi_n({\rm d}\bfU) ,
\end{equation}
which concludes the proof of the first assertion of \Cref{propVariancePosteriorBoundMatrix}. The second assertion follows from the first one by choosing $\barbfB = \hatbfB$.
\end{proof}

\begin{lemma}\label{propfondMatrix}
In the event $\mce_\gamma=\{\| \bxi^{\top}\bbfX\| \le n\lambda/\gamma\}$, for any $\barbfB\in\bbfM$,we have
\begin{align}
\ell_n(\hatbfB,\bfBz)
    &\le \ell_n(\barbfB,\bfBz) +
    \frac{2\lambda}{\gamma} \big(\gamma\|\barbfB\|_1-\gamma\|\hatbfB\|_1+\|\barbfB-\hatbfB\|_1\big)-\|\barbfB-\hatbfB\|_{L_2(\bbfX)}^2+2 H(\tau).
\end{align}
\end{lemma}

\begin{proof}
On the one hand, using the definitions of the prediction loss $\ell_n$ and the empirical loss $L_n$,
as well as the Von Neumann inequality, we get
\begin{align}
\ell_n(\hatbfB,\bfBz) - \ell_n(\barbfB,\bfBz)
        & =  2(V_n(\hatbfB) - V_n(\barbfB))  + \frac2n\, \sum_{i \in [n]}\xi_i \langle \bfX_i, \hatbfB-\barbfB \rangle + 2\lambda(\|\barbfB\|_1-\|\hatbfB\|_1)\\
        &\le 2(V_n(\hatbfB) - V_n(\barbfB))  + \frac2n\, \|\bxi ^{\top}\bbfX\| \|\hatbfB-\barbfB\|_1+2\lambda(\|\barbfB\|_1-\|\hatbfB\|_1)\\
        &\stackrel{(\text{in }\mce_\gamma)}{\le} 2(V_n(\hatbfB) - V_n(\barbfB))  + 2\lambda(\|\barbfB\|_1-\|\hatbfB\|_1) +
            \frac{2\lambda}\gamma\, \|\barbfB-\hatbfB\|_1.\label{eq:7Matrix}
\end{align}
Notice that inequality \eqref{lemConcentrationMatrix} can be rewritten as
\begin{equation}\label{eq:5Matrix}
V_n(\hatbfB) \le V_n (\barbfB )  + \underbrace{V_n(\hatbfB) - \int_{\mcM} V_n \hat\pi_n +  m_1 m_2 \tau -
        \frac{1}{2} \int_{\mcM} \| \barbfB - \bfU \|_{L_2(\bbfX)}^{2}\,\hat\pi_n({\rm d}\bfU)}_{:=A}.
\end{equation}
One can check that
\begin{align}
V_n(\hatbfB)-\int_{\mcM} V_n \,\hat\pi_n
    &= \frac1{2}\|\hatbfB \|_{L_2(\bbfX)}^2+\lambda\|\hatbfB\|_1-\int_{\mcM} \Big(\frac1{2}\|\bfU \|_{L_2(\bbfX)}^2+
        \lambda\|\bfU\|_1\Big)\,\hat\pi_n({\rm d}\bfU),\\
\int \| \bfU - \barbfB \big\|_{L_2(\bbfX)}^2\,\hat\pi_n({\rm d}\bfU)
    &= \|\barbfB-\hatbfB \|_{L_2(\bbfX)}^2+\int  \| \bfU \|_{L_2(\bbfX)}^2\,\,\hat\pi_n({\rm d}\bfU)
        - \| \hatbfB \|_{L_2(\bbfX)}^2.
\end{align}
These inequalities, combined with the definition of $H$, given in \eqref{matrixH}, yield
\begin{equation}
A  = H(\tau)- \frac1{2}\|\barbfB-\hatbfB\|_{L_2(\bbfX)}^2.\label{eq:6Matrix}
\end{equation}
Inserting this inequality in \eqref{eq:5Matrix}  and using relation
\eqref{eq:7Matrix}, we get the claim of the lemma.
\end{proof}

The next step is to establish the counterpart of \Cref{lem:2} in the matrix setting.

\begin{lemma}\label{lem:2Matrix}
For every $J\in [\rk(\barbfB)]$, we have
\begin{align}
\frac{2\lambda}{\gamma} \big(\gamma\|\barbfB\|_1-\gamma\|\hatbfB\|_1+\|\barbfB-\hatbfB\|_1\big)-\|\barbfB-\hatbfB\|_{L_2(\bbfX)}^2
    &\le 4\lambda\|\calP_{\barbfB,J^c}(\barbfB)\|_{1} +
\frac{\lambda^{2}(\gamma+1)^{2} |J|}{\gamma^{2}\kappa_{\barbfB,J,(\gamma+1)/(\gamma-1)}}.
\end{align}
\end{lemma}

\begin{proof}
To ease notation, let us write $\barbfB_{J}$ and  $\barbfB_{J^c}$ instead of $\calP_{\barbfB,J}(\barbfB) =
\calP^\bot_{\barbfB,J^c}(\barbfB)$ and $\calP_{\barbfB,J^c}(\barbfB)$, respectively. Clearly, $\barbfB =
\barbfB_{J}+\barbfB_{J^c}$. Recall that $r=\rk(\barbfB)$ and $\barbfB = \bfV_1\bfSigma\bfV_2^\top$ is
the singular value decomposition of $\barbfB$. Note that the matrices  $\bfPi_{1,J^c} = \bfI_{m_1}-
\bfV_{1,J}\bfV_{1,J}^\top$  and $\bfPi_{2,J^c} = \bfI_{m_2}-\bfV_{2,J}\bfV_{2,J}^\top$ are orthogonal
projectors and, for every matrix $\bfU\in\mcM$, we have
$\calP_{\barbfB,J^c}(\bfU)= \bfPi_{1,J^c}\bfU\bfPi_{2,J^c}$.

Let $\bfW$ be a $m_1\times m_2$ matrix  such that $\|\bfW\|=1$ and
$\langle \calP_{\barbfB,J^c}(\hatbfB), \bfW\rangle = \|\calP_{\barbfB,J^c}(\hatbfB)\|_1$.
We set $\bfD = \bfV_{1,J}\bfV_{2,J}^\top + \bfPi_{1,J^c}\bfW\bfPi_{2,J^c}$. It is clear that
\begin{align}
\|\barbfB\|_1 \le \|\barbfB_J\|_1 + \|\barbfB_{J^c}\|_1  =  \langle \barbfB_{J},  \bfD\rangle
+ \|\barbfB_{J^c}\|_1
\end{align}
and, in view of the von Neumann inequality,  $\|\hatbfB\|_1 \ge \langle \hatbfB,  \bfD\rangle$. This implies that
\begin{align}\label{eq:15Matrix}
\|\barbfB\|_1 - \|\hatbfB\|_1\le \|\barbfB_{J^c}\|_1 + \langle \barbfB_{J}-\hatbfB,  \bfD\rangle.
\end{align}
As shown in \citep{KLT11},
$\langle \barbfB_{J}-\hatbfB,  \bfD\rangle\le \|\calP_{\barbfB,J^c}^\bot(\barbfB-\hatbfB)\|_1-\|\calP_{\barbfB,J^c}(\hatbfB)\|_1$.
For the sake of self-containedness, we reproduce their proof here. We have
\begin{align}
\langle \barbfB_{J}-\hatbfB,  \bfD\rangle
    & = \langle \barbfB_{J}-\hatbfB,  \bfV_{1,J}\bfV_{2,J}+\bfPi_{1,J^c}\bfW\bfPi_{2,J^c}\rangle\\
    & = \langle \barbfB_{J}-\hatbfB,  \bfV_{1,J}\bfV_{2,J}^\top\rangle+
            \langle \bfPi_{1,J^c}(\barbfB_{J}-\hatbfB)\bfPi_{2,J^c}, \bfW\rangle\\
    & = \langle \barbfB-\hatbfB,  \bfV_{1,J}\bfV_{2,J}^\top\rangle -
            \langle \calP_{\barbfB,J^c}(\hatbfB), \bfW\rangle\\
    & = \langle \barbfB-\hatbfB,  \bfV_{1,J}\bfV_{2,J}^\top\rangle
             - \|\calP_{\barbfB,J^c}(\hatbfB)\|_1.
\end{align}
In addition, using the triangle inequality, we get $\|\calP_{\barbfB,J^c}(\hatbfB)\|_1
\ge \|\calP_{\barbfB,J^c}(\barbfB-\hatbfB)\|_1-\|\barbfB_{J^c}\|_1$. Thus,
we get
\begin{align}
\langle \barbfB_{J}-\hatbfB,  \bfD\rangle
    & \le \langle \barbfB-\hatbfB,  \bfV_{1,J}\bfV_{2,J}^\top\rangle  - \|\calP_{\barbfB,J^c}(\barbfB-\hatbfB)\|_1
        +\|\barbfB_{J^c}\|_1.\label{eq:16matrix}
\end{align}
Finally, one easily checks that $\langle \barbfB-\hatbfB,\bfV_{1,J}\bfV_{2,J}^\top\rangle
= \langle \calP_{\barbfB,J^c}^\bot(\barbfB-\hatbfB),\bfV_{1,J}\bfV_{2,J}^\top\rangle \le
\|\calP_{\barbfB,J^c}^\bot(\barbfB-\hatbfB)\|_1$. Combining this inequality with \eqref{eq:15Matrix} and
\eqref{eq:16matrix}, we get
\begin{align}
\|\barbfB\|_1 - \|\hatbfB\|_1
    \le 2\|\barbfB_{J^c}\|_1 + \|\calP_{\barbfB,J^c}^\bot(\barbfB-\hatbfB)\|_1
        -\|\calP_{\barbfB,J^c}(\barbfB-\hatbfB)\|_1.\label{eq:17Matrix}
\end{align}
If we set $\bfM = \barbfB-\hatbfB$, then we have already shown that
\begin{align}
\frac{2\lambda}{\gamma} \big\{\gamma\|\barbfB\|_1 &- \gamma\|\hatbfB\|_1+\|\barbfB-\hatbfB\|_1\Big\}-\|\barbfB-\hatbfB\|_{L_2(\bbfX)}^2\\
    &\le 4\lambda\|\barbfB_{J^c}\|_1 + \frac{2\lambda}{\gamma}\big(\gamma\|\calP_{\barbfB,J^c}^\bot(\bfM)\|_1
        -\gamma\|\calP_{\barbfB,J^c}(\bfM)\|_1 + \|\bfM\|_1\big)- \|\bfM\|_{L_2(\bbfX)}^2.
\end{align}
We remark that
\begin{eqnarray}
\gamma\|\calP_{\barbfB,J^c}^\bot(\bfM)\|_1 -\gamma\|\calP_{\barbfB,J^c}(\bfM)\|_1 + \|\bfM\|_1
    \le (\gamma+1)\|\calP_{\barbfB,J^c}^\bot(\bfM)\|_1 - (\gamma-1)\|\calP_{\barbfB,J^c}(\bfM)\|_1.
\label{OI1e3Matrix}
\end{eqnarray}
Now, by definition of the compatibility factor $\kappa_{\barbfB,J,c}$ given by equation \eqref{CFMatrix}, we obtain
\begin{equation}
\label{OI1e6Matrix}
\|\calP_{\barbfB,J^c}^\bot(\bfM)\|_1 - \frac{\gamma-1}{\gamma+1} \|\calP_{\barbfB,J^c}(\bfM)\|_1
\le \bigg(\frac{ |J|\,\| \bfM \|^{2}_{L_2(\bbfX)}}{n\kappa_{\barbfB,J,(\gamma+1)/(\gamma-1)}}\bigg)^{1/2}.
\end{equation}
Hence, inequalities \eqref{OI1e3Matrix} end \eqref{OI1e6Matrix} imply that
\begin{equation}
\frac{2\lambda}{\gamma}\big(\gamma\|\calP_{\barbfB,J^c}^\bot(\bfM)\|_1
        -\gamma\|\calP_{\barbfB,J^c}(\bfM)\|_1 + \|\bfM\|_1\big)- \|\bfM\|_{L_2(\bbfX)}^2
        \le  2ab -a^{2},
\label{OI1e7Matrix}
\end{equation}
where we have used the notation $a^2 = \| \bfM\|_{L_2(\bbfX)}^2$ and
$b^2 = \frac{\lambda^2(\gamma+1)^2|J|}{\gamma^2\kappa_{\barbfB,J,(\gamma+1)/(\gamma-1)}}$.
Finally, noticing that $2ab -a^{2}\le b^2$ we get the claim of the lemma.
\end{proof}

Combining the claims of the previous lemmas and taking the minimum with respect to $J$ and $\barbfB$,
we obtain that the inequality
\begin{equation}\label{eq:8Matrix}
\ell_{n} \big( \hatbfB,\bfBz \big) \le \inf_{\substack{\barbfB \in \mcM \\ J\subset[\rk(\barbfB)]}}
\bigg\{\ell_{n}\big(\barbfB,\bfBz \big)+4\lambda\|\calP_{\barbfB,J^c}(\barbfB)\|_{1} +
\frac{\lambda^{2}(\gamma+1)^{2} |J|}{\gamma^{2}\kappa_{\barbfB,J,(\gamma+1)/(\gamma-1)}}\bigg\}+2H(\tau)
\end{equation}
holds in the event $\mce_\gamma$. At this point, we remark that we have proved the more general result of Inequality \eqref{SOI2Matrix}.

The third and last step of the proof consists in assessing the probability of the event $\mce_\gamma$. We rely on Theorem 4.1.1
from \citep{tropp2015introduction} that provides a comprehensive account on matrix concentration inequalities.

\begin{lemma}\label{lem:3Matrix}
Let $\bbfX$ be a fixed design tensor and $v_\bbfX$ be defined by \eqref{vX}. If $\bxi\sim \mathcal N(\mathbf 0_n,\sigma^{2} \bfI_n)$,
then, for all $\e>0$,
\begin{equation}
    \prob\big(\| \bxi ^{\top} \bbfX \| > n \e\big)\le (m_1 + m_2)\,\exp\big(-{n\e^{2}}/{(2\sigma^{2}v_\bbfX^2)}\big).
\end{equation}
\end{lemma}

\begin{proof}
It is clear that $\xi_i/\sigma$ are standard gaussian random variables. Therefore, we can apply
\citep[Theorem 4.1.1]{tropp2015introduction} to the $m_1\times m_2$ matrix
\begin{equation}
 \bfZ = \sum_{i =1}^n\xi_i \bfX_i / \sigma.
\end{equation}
One easily checks that
\begin{align}
     v(\bfZ)    & = \| \mathbb{E} (\bfZ \bfZ^{\top}) \|\vee \| \mathbb{E} (\bfZ^{\top} \bfZ)\| \\
                & =  \bigg\| \sum_{i =1}^n \bfX_i \bfX_i^{\top} \bigg\|\bigvee \bigg\| \sum_{i =1}^n \bfX_i ^{\top} \bfX_i \bigg\|
                 = v_\bbfX.
\end{align}
Therefore,
\begin{equation}
 \prob\big(\| \bxi ^{\top} \bbfX / \sigma \| > n\e / \sigma \big)\le (m_1 + m_2)\,\exp\big(-{n\e^{2}}/(2\sigma^2v_\bbfX)\big),
\end{equation}
from which we deduce the claim of \Cref{lem:3Matrix}.
\end{proof}

A proof of \Cref{OI1Matrix} can be deduced from the three previous lemmas as follows. Choosing $\gamma=2$ and
$\e ={\lambda/\gamma} \ge\sigma v_{\calX}\sqrt{\nicefrac2n\log((m_1 + m_2)/\delta)}$ in \Cref{lem:3Matrix}, we get that
the event $\mce_\gamma$ has a probability at least $1-\delta$. Furthermore, on this event, we have already
established inequality \eqref{eq:8Matrix}, which coincides with the claim of \Cref{OI1Matrix}.

We conclude this section by proving \Cref{theoPosteriorConcentrationMatrix} which is the analogue
of \Cref{theoPosteriorConcentration}.
Let us introduce the (random) set $\mcb = \{\bfB\in\bbfM: V_n(\bfB) \le \int V_n\,\hat\pi_n + m_1 m_2\tau\}$.
Applying \eqref{eqBobkovMatrix} with $t=\sqrt{m_1 m_2}$, we get $\hat\pi_n(\mcb)\ge 1-2e^{-\sqrt{m_1 m_2}/16}$.
To prove \Cref{theoPosteriorConcentrationMatrix}, it is sufficient to check that in the event $\mce_\gamma$
(in particular, with $\gamma=2$), every matrix $\bfB$ from $\mcb$ satisfies the inequality
\begin{equation}\label{eq:9Matrix}
\ell_{n}(\bfB,\bfBz ) \le \inf_{\substack{\barbfB\in\mcM \\ J\in [\rk(\barbfB)]}}
    \bigg\{\ell_{n}(\barbfB,\bfBz )+ 4\lambda\|\calP_{\barbfB,J^c}(\barbfB)\|_{1}
    +\frac{9\lambda^{2}|J|}{2\kappa_{\barbfB,J,3}}\bigg\} + 8m_1m_2\tau.
\end{equation}
In the rest of this proof, $\bfB$ is always a matrix from $\mcb$. In view of \eqref{lemConcentrationMatrix}, it satisfies
\begin{equation}\label{eq:10Matrix}
V_n(\bfB) \le 2 m_1 m_2\tau + V_n(\barbfB)-
\frac{1}{2} \int_{\mcM} \|\bfU - \barbfB \|_{L_2(\bbfX)}^{2}\,\hat\pi_n({\rm d}\bfU).
\end{equation}
Note that \eqref{eq:10Matrix} holds for every $\barbfB\in\bbfM$. Therefore, it also holds for $\barbfB=\bfB$ and yields
\begin{equation}\label{eq:13Matrix}
 \int_{\mcM} \| \bfU - \bfB \|_{L_2(\bbfX)}^{2}\,\hat\pi_n({\rm d}\bfU) \le  4 m_1 m_2\tau.
\end{equation}
In addition, we have
\begin{equation}\label{eq:11Matrix}
\ell_n(\bfB,\bfBz) - \ell_n(\barbfB,\bfBz)  =  2\big(V_n(\bfB) - V_n(\barbfB) \big) +
    \frac2n\, \sum_{i=1}^n \xi_i \langle \bfX_i, \bfB-\barbfB \rangle + 2\lambda( \| \barbfB \|_{1} - \| \bfB \|_{1} ).
\end{equation}
Combining \eqref{eq:10Matrix}, \eqref{eq:11Matrix} and the Von Neuman inequality, we get that in $\mce_\gamma$
\begin{align}
\ell_n(\bfB,\bfBz) - \ell_n(\barbfB,\bfBz) & \le  4 m_1 m_2 \tau + \frac{2\lambda}{\gamma}\big(\gamma\| \barbfB \|_{1}
        - \gamma\| \bfB \|_{1} +\|\bfB-\barbfB\|_1 \big)\\
        &\qquad -  \int_{\mcM} \| \bfU - \barbfB \|_{L_2(\bbfX)}^{2}\,\hat\pi_n({\rm d}\bfU).\label{eq:12Matrix}
\end{align}
We use now the inequality $\|\bfU-\barbfB\|_{L_2(\bbfX)}^2\ge \nicefrac12\|\bfB-\barbfB\|_{L_2(\bbfX)}^2-\|\bfU-\bfB\|_{L_2(\bbfX)}^2$,
in conjunction with \eqref{eq:13Matrix}, to deduce  from \eqref{eq:12Matrix}  that
\begin{align}
\ell_n(\bfB,\bfBz) - \ell_n(\barbfB,\bfBz)
    & \le  8 m_1 m_2\tau + \frac{2\lambda}{\gamma}\,\|\bfB-\barbfB\|_1 + 2\lambda( \| \barbfB \|_{1} -
    \| \bfB \|_{1} ) - \frac{1}{2} \| \bfB -\barbfB \|_{L_2(\bbfX)}^{2}.
\end{align}
We can apply now \Cref{lem:2Matrix} with $\bfB$ instead of $\hatbfB$ and $\bbfX/\sqrt{2}$ instead of $\bbfX$ in order to get
the claim of \Cref{theoPosteriorConcentrationMatrix}.


\section*{Acknowledgments}
The work of Q.~Paris was supported by the Russian Academic Excellence Project 5-100. The work of A.~Dalalyan was partially
supported by the grant Investissements d'Avenir (ANR-11-IDEX-0003/Labex Ecodec/ANR-11-LABX-0047) and the chair
``LCL/GENES/Fondation du risque, Nouveaux enjeux pour nouvelles donn\'ees''.

\setlength{\bibsep}{3pt}

\bibliography{bibDGP16}
\end{document}